\documentclass{amsart}
\usepackage[backref=page]{hyperref}
\usepackage[alphabetic]{amsrefs}
\usepackage{tikz}\usetikzlibrary{cd}
\usepackage{booktabs,amssymb}

\title
[Non-taut RDPs and the height of K3 surfaces]
{Inseparable maps on $W_n$-valued local cohomology groups of non-taut rational double point singularities 
and the height of K3 surfaces}

\author{Yuya Matsumoto}
\address{Department of Mathematics, Faculty of Science and Technology, Tokyo University of Science, 2641 Yamazaki, Noda, Chiba, 278-8510, Japan}
\email{\url{matsumoto.yuya.m@gmail.com}}
\email{\url{matsumoto_yuya@ma.noda.tus.ac.jp}}
\thanks{This work was supported by JSPS KAKENHI Grant Number 16K17560 and 20K14296.}
\date{2022/03/22}

 \theoremstyle{plain}
 \newtheorem{thm}{Theorem}[section]
 \newtheorem{lem}[thm]{Lemma}
 \newtheorem{prop}[thm]{Proposition}
 \newtheorem{cor}[thm]{Corollary}
 \theoremstyle{definition}
 \newtheorem{rem}[thm]{Remark}
 \newtheorem{defn}[thm]{Definition}
 \newtheorem{example}[thm]{Example}
 \newtheorem{conv}[thm]{Convention}
 \newtheorem{setting}[thm]{Setting}

\newcommand{\cO}{\mathcal O}
\newcommand{\cI}{\mathcal I}
\newcommand{\cX}{\mathcal X}
\newcommand{\rI}{\mathrm I}
\newcommand{\rII}{\mathrm {II}}

\newcommand{\floor}[1]{\lfloor #1 \rfloor}
\newcommand{\abs}[1]{\lvert #1 \rvert}
\newcommand{\card}[1]{\lvert #1 \rvert}
\newcommand{\localcoh}[3][2]{H^{#1}_{#2}(#3)}
\newcommand{\localcohtorsion}[4][2]{\localcoh[#1]{#2}{#3}[#4]}
\newcommand{\map}[4][\to]{#2 \colon #3 #1 #4}
\newcommand{\namedto}[1]{\xrightarrow{#1}}
\newcommand{\injto}{\hookrightarrow}
\newcommand{\isomto}{\stackrel{\sim}{\to}}
\newcommand{\isomfrom}{\stackrel{\sim}{\leftarrow}}
\newcommand{\restrictedto}[1]{\rvert_{#1}}
\newcommand{\set}[1]{\{#1\}}
\newcommand{\cat}[1]{\{#1\}}
\newcommand{\bP}{\mathbb P}
\newcommand{\bN}{\mathbb N}
\newcommand{\bZ}{\mathbb Z}
\newcommand{\bQ}{\mathbb Q}
\newcommand{\bC}{\mathbb C}
\newcommand{\bF}{\mathbb F}
\newcommand{\bG}{\mathbb G}
\newcommand{\Gm}{\mathrm{\bG_m}}
\newcommand{\idealm}{\mathfrak{m}}
\newcommand{\idealp}{\mathfrak{p}}
\newcommand{\idealq}{\mathfrak{q}}
\newcommand{\thpower}[2]{{#1}^{(#2)}}
\newcommand{\pthpower}[1]{\thpower{#1}{p}}
\DeclareMathOperator{\Km}{Km}
\DeclareMathOperator{\Ann}{Ann}
\DeclareMathOperator{\height}{ht}
\DeclareMathOperator{\Sing}{Sing}
\DeclareMathOperator{\Spec}{Spec}
\DeclareMathOperator{\Ext}{Ext}
\DeclareMathOperator{\Supp}{Supp}
\DeclareMathOperator{\Ker}{Ker}
\DeclareMathOperator{\Coker}{Coker}
\DeclareMathOperator{\Bl}{Bl}
\DeclareMathOperator{\Fix}{Fix}
\DeclareMathOperator{\Exc}{Exc}
\DeclareMathOperator{\Br}{Br}
\DeclareMathOperator{\charac}{char}
\DeclareMathOperator{\Proj}{Proj}
\DeclareMathOperator{\Frac}{Frac}
\DeclareMathOperator{\rank}{rank}
\DeclareMathOperator{\Pic}{Pic}
\DeclareMathOperator{\End}{End}
\DeclareMathOperator{\disc}{disc}
\DeclareMathOperator{\Frob}{Frob}
\DeclareMathOperator{\Image}{Im}
\DeclareMathOperator{\Hom}{Hom}
\DeclareMathOperator{\ord}{ord}
\newcommand{\et}{\mathrm{\acute et}}
\newcommand{\crys}{\mathrm{crys}}
\newcommand{\Het}{H_\et}
\newcommand{\Hcrys}{H_\crys}
\newcommand{\Witt}[2][]{W_{#1}(#2)}
\newcommand{\functorspace}{{\mathord{-}}}
\newcommand{\WO}{W \mathcal{O}}
\newcommand{\divides}{\mid}
\newcommand{\notdivides}{\nmid}
\newcommand{\Legendre}[2]{(\frac{#1}{#2})}

\newcommand{\negligible}[1]{\{#1\}}
\newcommand{\gap}{\delta}

\begin{document}

\begin{abstract}
We consider rational double point singularities (RDPs) that are non-taut, 
which means that the isomorphism class is not uniquely determined from the dual graph of the minimal resolution.
Such RDPs exist in characteristic $2,3,5$.
	We compute the actions of Frobenius, and other inseparable morphisms,
	on $W_n$-valued local cohomology groups of RDPs.
	Then we consider RDP K3 surfaces admitting non-taut RDPs.
	We show that the height of the K3 surface, which is also defined in terms of the Frobenius action on $W_n$-valued cohomology groups,
	is related to the isomorphism class of the RDP.
\end{abstract}

\maketitle

\section{Introduction} \label{sec:introduction}

\subsection{Non-taut rational double points}

In this subsection we recall non-taut rational double points.

Rational double point singularities (RDPs for short) are the simplest normal singularities in dimension $2$.
The fundamental invariant of an RDP 
is the dual graph of the exceptional divisor of the minimal resolution,
which is a Dynkin diagram.
In most cases the dual graph determines the isomorphism class of the singularity (in a fixed characteristic $p \geq 0$). 
Such RDPs are called \emph{taut}. 
However in some special cases there are more than one isomorphism classes,
in which cases the RDPs are called \emph{non-taut}.

To describe them we define, 
for each pair of a characteristic $p \geq 0$ and a Dynkin diagram $S$ (which is $A_N$, $D_N$, or $E_N$),
a non-negative integer $r_{\max} = r_{\max}(S) = r_{\max}(p,S)$ as follows:
\[
r_{\max}(p,S) = \begin{cases}
\floor{N/2}-1 & \text{if $(p,S) = (2,D_N)$,} \\ 
1             & \text{if $(p,S) = (2,E_6)$,} \\
3             & \text{if $(p,S) = (2,E_7)$,} \\
4             & \text{if $(p,S) = (2,E_8)$,} \\
1             & \text{if $(p,S) = (3,E_6), (3,E_7)$,}\\
2             & \text{if $(p,S) = (3,E_8)$,} \\
1             & \text{if $(p,S) = (5,E_8)$,} \\
0             & \text{otherwise.}
\end{cases}
\]

\begin{thm}[Artin \cite{Artin:RDP}] \label{thm:Artin}
	Let $p$ and $S$ as above.
	Then there exist exactly $r_{\max} + 1$ isomorphism classes of RDPs in characteristic $p$ whose dual graph is a Dynkin diagram of type $S$.
	Explicit equations are given and,
	when $r_{\max} > 0$, the isomorphism classes are distinguished by the symbols $S^r$ ($0 \leq r \leq r_{\max}$), see Table \ref{table:non-taut RDP} in Section \ref{sec:RDP}.
\end{thm}
The isomorphism classes $S^r$ are ordered in the way that the coindex $r$ is lower semi-continuous in families of RDPs with the same dual graph $S$.

\subsection{Inseparable maps on $W_n$-valued local cohomology groups of non-taut RDPs}

In this paper we consider 
$W_n$-valued local cohomology groups $\localcoh{\idealm_A}{\Witt[n]{A}}$, 
and their $I$-torsion parts $\localcohtorsion{\idealm_A}{\Witt[n]{A}}{I}$ for ideals $I \subset A$,
of (mainly non-taut) RDPs $A$ for some $n$.
We compute the Frobenius actions on certain classes of such cohomology groups.
See Section \ref{sec:RDP:Frobenius} for precise results.
The behaviors of the Frobenius actions depend heavily on the isomorphism classes of the non-taut RDP,
and we can derive Theorem \ref{thm:main:introduction} from these results.

More thorough studies concerning kernels of powers of Frobenius
can distinguish isomorphism classes among the same dual graph, as will be proven in a subsequent work \cite{Liedtke--Martin--Matsumoto:RDPtors} by Liedtke, Martin, and the author.
Also we refer to Tanaka's work \cite{Tanaka:taut}
concerning relations of tautness of $2$-dimensional singularities 
and other kind of Frobenius-related properties, such as $F$-regularity and $F$-purity.

We also compute (Section \ref{sec:RDP:mu alpha}) the maps 
induced by $\mu_p$- or $\alpha_p$-quotient morphisms,
which are another kind of inseparable morphisms.
This is used to show Theorem \ref{thm:quotient:introduction}.

\subsection{Height of K3 surfaces with non-taut RDPs: main results}

The \emph{height} is an invariant of K3 surfaces in positive characteristic 
which takes values in $\set{1, 2, \dots, 10} \cup \set{\infty}$.
Among several characterizations, the most relevant to our purpose is the one by the Frobenius actions on $W_n$-valued cohomology groups (see Theorem \ref{thm:vdGK}).
Proposition \ref{prop:local to height:bis} connects this and the computations of Frobenius actions on $W_n$-valued local cohomology groups of RDPs,
and we obtain the following result on the height of RDP K3 surfaces 
(by which we mean a proper surface with only RDP singularities
whose minimal resolution is a K3 surface in the usual sense).

\begin{thm}[see Theorem \ref{thm:main} for a detailed statement] \label{thm:main:introduction}
	For each Dynkin diagram $S$ and characteristic $p$ (with $r_{\max}(p,S) > 0$),
	we give a subsequence $(r_1, r_2, \dots, r_l)$ of $(r_{\max}(p,S), \dots, 2, 1)$
	with the following properties.
	Suppose an RDP K3 surface $Y$ in characteristic $p$ 
	admits an RDP of type $S^r$.
	\begin{itemize}
		\item If $r > 0$, then $\height(Y) \leq l$ and $r = r_{\height(Y)}$.
		\item If $r = 0$, then $\height(Y) > l$.
	\end{itemize}
	
	In short, $\height(Y)$ determines $r$, 
	and if $r > 0$ then $r$ determines $\height(Y)$.
\end{thm}
If $(p,S)$ is not $(2, D_N)$ ($N \geq 8$) nor $(2, E_8)$, 
then the subsequence is the entire sequence $(r_{\max}(p, S), \dots, 2, 1)$.

\medskip

While the index of an RDP on an RDP K3 surface is bounded by the Picard number of the minimal resolution and hence by $22$,
this theorem shows the non-existence of certain RDPs (e.g.\ $E_8^1$) in characteristic $2$ with a different reason.
In Section \ref{sec:realizable}, 
we determine which non-taut RDPs are realizable on RDP K3 surfaces (Theorem \ref{thm:existence:non-taut}).

\medskip

Now suppose $\map{\pi}{X}{Y}$ is a $G$-quotient morphism between RDP K3 surfaces,
$G \in \set{\mu_p, \alpha_p, \bZ/p\bZ}$. 
If $G = \mu_p$ or $G = \alpha_p$, 
then the ``dual'' map $\map{\pi'}{\thpower{Y}{1/p}}{X}$ is also a $G'$-quotient with $G' = \mu_p$ or $G' = \alpha_p$,
and we can use the local behavior of the pullback maps by $\pi$ and $\pi'$ 
to relate the singularities of $X$ and $Y$ to the height of $X$ and $Y$.
If $G = \bZ/p\bZ$, then $Y$ always have a non-taut RDP, to which we can apply Theorem \ref{thm:main:introduction}.
Thus we obtain the following.

\begin{thm}[see Theorems \ref{thm:quotient} and \ref{thm:quotient:etale} for a detailed statement] \label{thm:quotient:introduction}
Let $\map{\pi}{X}{Y}$ be a $G$-quotient morphism between RDP K3 surfaces,
where $G \in \set{\mu_p, \alpha_p, \bZ/p\bZ}$. 
Then we have $\height(X) = \height(Y) =: h$, 
we determine $h$ in terms of $\Sing(Y)$ and $\Sing(X)$,
and $h$ is always finite.
\end{thm}

As a consequence, 
we prove (Corollary \ref{cor:criterion}) that $G$-quotient of an RDP K3 surface $X$ in characteristic $p$, with $G = \mu_p$ or $G = \alpha_p$,
is an RDP K3 surface if and only if $X$ is of finite height.

\subsection{Organization of the paper}

In Section \ref{sec:Witt} we recall the definition and basic properties of the rings $\Witt[n]{A}$ of truncated Witt vectors. 
In Section \ref{sec:Witt local cohomology} we introduce morphisms between $W_n$-valued local cohomology groups and interpret them in terms of Cech cohomology groups.
In Section \ref{sec:RDP} we carry out explicit computations for inseparable morphisms between RDPs. 
In Section \ref{sec:height} we recall the definition and basic properties of the height of K3 surfaces.
The main results, connecting the height of K3 surfaces
and the maps on $W_n$-valued local cohomology groups, will be proved in Section \ref{sec:height of morphisms}.
In Sections \ref{sec:realizable} and \ref{sec:examples} 
we discuss which RDPs are realizable on RDP K3 surfaces,
and give examples for all possible non-taut RDPs.

\section{Rings of truncated ($p$-typical) Witt vectors} \label{sec:Witt}

We recall the definition and basic properties of the rings $\Witt[n]{A}$ of truncated $p$-typical Witt vectors. 

Let $p$ be a prime and $A$ an $\bF_p$-algebra. 
The ring $W(A)$ of \emph{$p$-typical Witt vectors} on $A$ is the set $A^{\bN}$
equipped with the ring structure satisfying,
for each polynomial $P \in \bZ[x,y]$,
\[
P( (a_0, a_1, \dots), (b_0, b_1, \dots) )
= ( P_0(a_0, b_0), P_1(a_0, b_0, a_1, b_1), \dots),
\]
where $P_i \in \bZ[x_0, \dots, x_{i-1}, y_0, \dots, y_{i-1}]$
is the unique collection of polynomials 
satisfying, for each $N \in \bN$,
\[
w_N(P_0(\dots), P_1(\dots), \dots, P_{N}(\dots))
= P(w_N(a_0, a_1, \dots, a_N), w_N(b_0, b_1, \dots, b_N)),
\]
where $w_N(t_0, t_1, \dots, t_N) := \sum_{i=0}^{N} p^i t_i^{p^{N-i}}$
is the so-called $N$-th ghost component.

For example, we clearly have $(a_0) + (b_0) = (a_0 + b_0)$ and $(a_0) \cdot (b_0) = (a_0 b_0)$ on $\Witt[1]{A} \cong A$,
and it follows from the equalities 
\begin{align*}
(a_0^p + p a_1) + (b_0^p + p b_1) 
&= (a_0 + b_0)^p + p (a_1 + b_1 - \frac{(a_0 + b_0)^p - a_0^p - b_0^p}{p}), \\
(a_0^p + p a_1) \cdot (b_0^p + p b_1) 
&= (a_0 b_0)^p + p (a_1 b_0^p + a_0^p b_1 + p a_1 b_1) 
\end{align*}
that
\begin{align*}
(a_0, a_1) + (b_0, b_1) &= (a_0 + b_0, a_1 + b_1 - \frac{(a_0 + b_0)^p - a_0^p - b_0^p}{p}), \\
(a_0, a_1) \cdot (b_0, b_1) &= (a_0 b_0, a_1 b_0^p + a_0^p b_1 + p a_1 b_1)
\end{align*}
on $\Witt[2]{A}$.

It follows from \cite{Hazewinkel:formal groups}*{17.1.18} that, for $P(x,y) = x + y$,
the $i$-th component $P_i$ of $P((a_0, a_1, \dots), (b_0, b_1, \dots))$ 
is a homogeneous polynomial of $a_0, a_1, \dots, b_0, b_1, \dots$ of degree $p^i$ 
if we declare $a_i$ and $b_i$ to be homogeneous of degree $p^i$.

$W$ is functorial: any homomorphism $\map{f}{A}{B}$ of $\bF_p$-algebras 
induces a morphism $\map{f}{W(A)}{W(B)}$ of rings by $f(a_0, a_1, \dots) = (f(a_0), f(a_1), \dots)$
that is compatible with $V$ and $R$ defined below.
An example is the \emph{Frobenius} morphism $\map{F}{W(A)}{W(A)}$ defined as 
$F(a_0, a_1, \dots) = (a_0^p, a_1^p, \dots)$.

The shift morphism, or \emph{Verschiebung}, $V$ on $W(A)$ is defined as 
$V(a_0, a_1, \dots) = (0, a_0, a_1, \dots)$.

The ring of \emph{Witt vectors of length $n$}
is the quotient $\Witt[n]{A} = W(A) / V^n W(A)$,
hence in $\Witt[n]{A}$ only the first $n$ components $(a_0, a_1, \dots, a_{n-1})$ are considered.
The Verschiebung induces $\map{V}{\Witt[n]{A}}{\Witt[n+1]{A}}$.
The restriction morphism $\map{R}{\Witt[n]{A}}{\Witt[n-1]{A}}$ is defined as 
$R(a_0, \dots, a_{n-1}) = (a_0, \dots, a_{n-2})$ 
and is a ring homomorphism.
We have an exact sequence
\[
0 \to \Witt[n']{A} \namedto{V^{n-n'}} \Witt[n]{A} \namedto{R^{n'}} \Witt[n-n']{A} \to 0
\]
for each $0 \leq n' \leq n$.

We use the following equalities in Section \ref{sec:RDP}. 

\begin{lem} \label{lem:projection formula}
If $x \in \Witt[n]{A}$ and $y \in \Witt[n+m]{A}$,
then $V^m(x) \cdot y = V^m(x \cdot F^m(R^m(y)))$.
\end{lem}

\begin{lem} \label{lem:subtraction}
	Let $A$ be an $\bF_p$-algebra.
	\begin{enumerate}
	\item \label{item:subtraction:2:n}
	In $\Witt[n]{A}$ with $p = 2$,
	write $(t_1+t_2, 0, 0, \dots) - (t_2, 0, 0, \dots) = (S_0(t_1,t_2), S_1(t_1,t_2), \dots)$
	with polynomials $S_i \in \bF_2[t_1,t_2]$.
	Then $S_i$ is homogeneous of degree $2^i$
	and we have $S_i \equiv t_1 t_2^{2^i-1} \pmod{t_1^2}$.
	\item \label{item:subtraction:2:3}
	In $\Witt[3]{A}$ with $p = 2$, we have
\[
\begin{pmatrix} a + b + c \\ 0 \\ 0 \end{pmatrix} - 
\begin{pmatrix} b \\ b c \\ 0 \end{pmatrix} 
= \begin{pmatrix} a + c \\ a b \\
(a + c)^3 b + (a + c) b^3 + (a^2 + 3 a c + c^2) b^2 \end{pmatrix}.
\]
	\item \label{item:subtraction:2:4}
	In $\Witt[4]{A}$ with $p = 2$, we have
	\[
	\begin{pmatrix} a + b \\ 0 \\ 0 \\ 0 \end{pmatrix} - 
	\begin{pmatrix} a     \\ 0 \\ 0 \\ 0 \end{pmatrix} - 
	\begin{pmatrix}     b \\ 0 \\ 0 \\ 0 \end{pmatrix} =
	\begin{pmatrix} 0 \\ a b \\ a b (a^2 + a b + b^2) \\ a b (a^6 + a^5 b + a^3 b^3 + a b^5 + b^6) \end{pmatrix}
	\]
	and
	\[
	(c_0, c_1, c_2 + d_2, c_3) - (0, 0, d_2, 0) = (c_0, c_1, c_2, c_3 + c_2 d_2).
	\]
	\item \label{item:subtraction:3:2}
	In $\Witt[2]{A}$ with $p = 3$, we have
	\[
		(a + b, 0) - (a, 0) - (b, 0) = (0, a b (a + b)).
	\]
	\item \label{item:subtraction:5:2}
	In $\Witt[2]{A}$ with $p = 5$, we have
	\[
	(a + b, 0) - (a, 0) - (b, 0) = (0, a b (a + b) (a^2 + a b + b^2)).
	\]
	\end{enumerate}
\end{lem}
\begin{proof}
	Straightforward.
\end{proof}

The closed immersion $\map{R^*}{\Spec \Witt[n-1]{A}}{\Spec \Witt[n]{A}}$
is a homeomorphism if $n \geq 2$.
For an $\bF_p$-scheme $Z$ and $n \geq 1$, we define $\Witt[n]{Z}$ to be the scheme
whose underlying topological space is $Z$ 
and whose structure sheaf is $\Witt[n]{\cO_Z}$.

\begin{lem} \label{lem:projective}
If $Z$ is a scheme projective (resp.\ quasi-projective) over an algebraically closed field $k$,
then $\Witt[n]{Z}$ is projective (resp.\ quasi-projective) over $\Witt[n]{k}$.
\end{lem}

\begin{proof}
Since $\Witt[n]{\functorspace}$ preserves open immersions and closed immersions,
it suffices to show that $\Witt[n]{\bP^N}$ is projective.
We will show that $\Witt[n]{k[x_0, \dots, x_N]}$ is a finitely generated $\Witt[n]{k}$-algebra.
Indeed, it is generated by the elements $(x_i, 0, \dots, 0)$ with $0 \leq i \leq N$
and the elements $V^j(x_0^{i_{0}} \dots x_N^{i_{N}}, 0, \dots, 0)$
with $0 < j < n$ and $0 \leq i_{k} < p^j$.
\end{proof}

\begin{lem} \label{lem:CM}
Suppose $(A, \idealm_A)$ is a Cohen--Macaulay local ring.
Then $\Witt[n]{A}$ is also Cohen--Macaulay.
More precisely, the Teichm\"uller lift of a regular sequence of $A$ is a regular sequence of $\Witt[n]{A}$.

In particular, $\Ext^i_{\Witt[n]{A}}(M', \Witt[n]{A}) = 0$
if $\Supp M' \subset \set{\idealm_A}$ and $i < \dim A$.
\end{lem}
\begin{proof}
The former assertion follows from \cite{Borger:WittII}*{Propositions 16.18, 16.19}.
(One can also show more generally that $b_1, \dots, b_N$ in $\Witt[n]{A}$
is a regular sequence 
if $R^{n-1}(b_1), \dots, R^{n-1}(b_N)$ is a regular sequence in $A$.)

The latter assertion is a consequence of being Cohen--Macaulay.
\end{proof}
Also note the equality $\Witt[n]{A[\frac{1}{x}]} = \Witt[n]{A}[\frac{1}{[x]}]$, where $[x]$ is the Teichm\"uller lift.

\section{$W_n$-valued local cohomology groups} \label{sec:Witt local cohomology}

Suppose $(A, \idealm_A)$ is a Noetherian Cohen--Macaulay local $k$-algebra of dimension $d$.
By \cite{Hartshorne:local cohomology}*{Theorems 2.8 and 3.8}, 
we have isomorphisms 
 $\varinjlim_{I} \Ext^i_{A}(A/I, A) \cong \localcoh[i]{\idealm_A}{A}$,
where the limit is taken over $\idealm_A$-primary ideals (i.e.\ $\Supp(A/I) \subset \set{\idealm_A}$), 
and we have $\localcoh[i]{\idealm_A}{A} = 0$ for $i < d$ and
\[ \localcoh[d]{\idealm_A}{A} \cong 
\Coker \Bigl( \bigoplus_{i = 1}^d A[(x_1 \dots \hat{x_i} \dots x_d)^{-1}] \to A[(x_1 \dots x_d)^{-1}] \Bigr) \]
for any regular sequence $x_1, \dots, x_d$ in $\idealm_A$.

\begin{lem} \label{lem:bis:interpretation of Ext}
Let $(A,\idealm_A)$ be as above. 
Let $n \geq 1$ be an integer and
$I \subset A$ an $\idealm_A$-primary ideal.
Let $J := (R^{n-1})^{-1}(I) \subset \Witt[n]{A}$,
so that $\Witt[n]{A}/J \cong A/I$.
Then the morphism \[
\map{h = h_I}
{\Ext^d_{\Witt[n]{A}}(A/I, \Witt[n]{A})}
{\localcoh[d]{\idealm_{\Witt[n]{A}}}{\Witt[n]{A}}}
\]
is injective, and its image is precisely the submodule $\localcohtorsion[d]{\idealm_{\Witt[n]{A}}}{\Witt[n]{A}}{J}$ of the classes annihilated by $J$.
Moreover this morphism is compatible with the inclusion of $\idealm_A$-primary ideals $I' \subset I$ and with $V$.
\end{lem}
\begin{proof}
Compatibility with $V$ is clear.

By replacing $A$ with $\Witt[n]{A}$, we may assume $n = 1$ (hence $J = I$).

Note that $\Ext^d_{A}(\functorspace, A)$ is a contravariant left exact functor from the category of $A$-modules of finite length.
Hence, for an inclusion $I' \subset I$, $\Ext^d_{A}(A/I, A) \to \Ext^d_{A}(A/I', A)$ is injective,
and this implies that $h_I$ is injective.

Suppose an element of $\localcoh[d]{\idealm_A}{A}$ has annihilator $I$ and is of the form $h(e)$ 
for $e \in \Ext^d_{A}(A/I', A)$. Then $I' \subset I$.
We want to show that $e$ comes from $\Ext^d_{A}(A/I, A)$.
Take a sequence $b_1, \dots, b_N$ generating $I$.
Applying $\Ext^d_{A}(\functorspace, A)$ to the exact sequence $\bigoplus_{i=1}^N A/I' \namedto{\bigoplus_i b_i} A/I' \to A/I \to 0$,
we obtain an exact sequence
\[
0 \to 
  \Ext^d_{A}(A/I, A) \to
  \Ext^d_{A}(A/I', A) \namedto{\bigoplus_i b_i}
  \bigoplus_i \Ext^d_{A}(A/I', A).
\]
Since $b_i e = 0$, it follows that $e$ comes from $\Ext^d_{A}(A/I, A)$.
\end{proof}

\begin{conv} \label{conv:WittExt}
For simplicity, we will write 
$\localcohtorsion[d]{\idealm_A}{\Witt[n]{A}}{I}$ instead of 
$\localcohtorsion[d]{\idealm_{\Witt[n]{A}}}{\Witt[n]{A}}{(R^{n-1})^{-1}(I)}$. 
\end{conv}

We have morphisms
\begin{align*}
 &\map{V}{\localcoh[d]{\idealm_A}{\Witt[n]{A}}}{\localcoh[d]{\idealm_A}{\Witt[n+1]{A}}}, 
	\\ 
 &\map{R}{\localcoh[d]{\idealm_A}{\Witt[n]{A}}}{\localcoh[d]{\idealm_A}{\Witt[n-1]{A}}}.
\end{align*}

	Suppose $B$ is another Noetherian Cohen--Macaulay local $k$-algebra of dimension $d$
	and $\map{f}{A}{B}$ is a local morphism (an example is the Frobenius morphism $\map{F}{A}{A}$).
	Then we have morphisms
	\[
		\map{f}{\localcoh[d]{\idealm_A}{\Witt[n]{A}}}{\localcoh[d]{\idealm_B}{\Witt[n]{B}}}, 
	\]

\begin{lem} \label{lem:morphisms on local cohomology}
The morphisms $f$, $V$, and $R$ induce morphisms
\begin{gather*}
\map{f}{\localcohtorsion[d]{\idealm_A}{\Witt[n]{A}}{I}}{\localcohtorsion[d]{\idealm_B}{\Witt[n]{B}}{IB}}, \\
\map{V}{\localcohtorsion[d]{\idealm_A}{\Witt[n]{A}}{I}}{\localcohtorsion[d]{\idealm_A}{\Witt[n+1]{A}}{I}}, \\
\map{R}{\localcohtorsion[d]{\idealm_A}{\Witt[n]{A}}{I}}{\localcohtorsion[d]{\idealm_A}{\Witt[n-1]{A}}{I}}. 
\end{gather*}
\end{lem}
\begin{proof}
Straightforward.
\end{proof}

For example, the Frobenius morphism induces 
\[
\map{F}{\localcohtorsion[d]{\idealm_A}{\Witt[n]{A}}{I}}{\localcohtorsion[d]{\idealm_A}{\Witt[n]{A}}{\pthpower{I}}}.
\]

\begin{lem}
Let $A$ be as above, and $x_1, \dots, x_d \in A$ a regular sequence of $A$.
Then the identification
\[
\localcoh[d]{\idealm_A}{\Witt[n]{A}} \cong 
\Coker \Bigl( \bigoplus_{i = 1}^d \Witt[n]{A[(x_1 \dots \hat{x_i} \dots x_d)^{-1}]} \to \Witt[n]{A[(x_1 \dots x_d)^{-1}]} \Bigr)
\]
is compatible with the morphisms $V$ and $R$.

If $\map{f}{A}{B}$ is as above,
and $f(x_1), \dots, f(x_d)$ is a regular sequence of $B$,
then the same assertion holds for $f$.
\end{lem}
\begin{proof}
Clear.
\end{proof}

\section{Computing morphisms on $W_n$-valued local cohomology groups on RDPs} \label{sec:RDP}

We carry out computations of Frobenius and other inseparable morphisms on $W_n$-valued local cohomology groups on RDPs,
which will be used in Section \ref{sec:height of morphisms} to prove the main theorems.
In Section \ref{sec:RDP:preparation} we fix notations.
In Section \ref{sec:RDP:Frobenius} (Propositions \ref{prop:local Frob:2:D}, \ref{prop:local Frob:2:E8}, \ref{prop:local Frob:E:bis})
we compute Frobenius morphisms.
The simplest is Proposition \ref{prop:local Frob:E:bis}(\ref{case:2:E6},\ref{case:3:E6E7},\ref{case:5:E8}), where only $\Witt[1]{A} = A$ is needed,
while the most complicated is Proposition \ref{prop:local Frob:2:D}, in which $D_N^r$ in characteristic $2$ is considered and $W_n$ with unbounded $n$ is needed.
In Section \ref{sec:RDP:mu alpha} (Proposition \ref{prop:local mu_p alpha_p}) we compute $\mu_p$- and $\alpha_p$-quotient morphisms.

\subsection{Equations of RDPs} \label{sec:RDP:preparation}

\begin{conv} \label{conv:equations of non-taut RDP}
	For each non-taut RDP $A$, 
	we often work under an isomorphism $A \cong k[[x,y,z]] / (f)$, 
	where $f$ is the polynomial as in Table \ref{table:non-taut RDP}.
	In the case of $D_N^0$ we use two different equations.
	All equations are taken from \cite{Artin:RDP},
	except for the equation with the term $ z x y^{m-r}$ in the case of $D_N^0$.
	\begin{table}
		\caption{Equations of non-taut RDPs} \label{table:non-taut RDP}
		\begin{tabular}{lllll}
			\toprule
			$p$ & & & equation & \\
			\midrule
			$2$ & $D_{2m}^r$   & $0 \leq r \leq m-1$ & $z^2 + x^2 y + x y^m + z x y^{m-r}$ & \\
			$2$ & $D_{2m}^0$   &                     & $z^2 + x^2 y + x y^m $              & \\
			$2$ & $D_{2m+1}^r$ & $0 \leq r \leq m-1$ & $z^2 + x^2 y + z y^m + z x y^{m-r}$ & \\
			$2$ & $D_{2m+1}^0$ &                     & $z^2 + x^2 y + z y^m $              & \\
			$2$ & $E_6^r$      & $r = 1,0$           & $z^2 + x^3   + y^2 z + b x y z$     & $b = 1,0$ \\
			$2$ & $E_7^r$      & $r = 3,2,1,0$       & $z^2 + x^3 + x y^3 + \beta$ & $\beta = z x y, z y^3, z x^2 y, 0$ \\
			$2$ & $E_8^r$      & $r = 4,3,2,1,0$     & $z^2 + x^3 + y^5 + \beta$ & $\beta = z x y, z y^3, z x y^2, z x y^3, 0$ \\
			\midrule
			$3$ & $E_6^r$ & $r = 1,0$   & $ z^2 + x^3   +   y^4   + b x^2 y^2  $ & $b = 1,0$ \\
			$3$ & $E_7^r$ & $r = 1,0$   & $ z^2 + x^3   + x y^3   + b x^2 y^2  $ & $b = 1,0$ \\
			$3$ & $E_8^r$ & $r = 2,1,0$ & $ z^2 + x^3   + y^5 + \lambda x^2 y^2$ & $\lambda = 1, y, 0$ \\
			\midrule
			$5$ & $E_8^r$ & $r = 1,0$   & $ z^2 + x^3   +   y^5   + (b/2) x y^4 $ & $b = 2,0$ \\
			\bottomrule
		\end{tabular}
	\end{table}
\end{conv}

\begin{setting} \label{setting:I}
Suppose a non-taut RDP $A$ and an integer $j \geq 1$ 
 satisfy one of the following, 
and define an ideal $I_j \subset A$ accordingly.
	\begin{itemize}
		\item $j = 1$, and $I_1 = \idealm_A \subset A$ is the maximal ideal.
		\item $1 \leq j \leq \floor{N/2}-1$, 
		$A$ is an RDP of type $D_{N}^r$ in characteristic $2$,
		and $I_j \subset A$ consists of the elements whose vanishing order at the $2j$-th component of 
		the exceptional divisor of the minimal resolution of $A$ is at least $2j$.
		\item $j = 2$, 
		$A$ is an RDP of type $E_8^r$ in characteristic $2$, 
		and $I_j \subset A$ consists of the elements whose vanishing order at the $4$-th component of 
		the exceptional divisor of the minimal resolution of $A$ is at least $8$.
	\end{itemize}
Here, in the case of $D_N$ or $E_8$, the components are ordered in a way that 
the $1$-st component is the end of the longest branch of the Dynkin diagram,
and the $i$-th component ($i \leq N - 2$ or $i \leq 5$ respectively)
is the unique component of distance $i-1$ from the $1$-st component.
(In the case of $D_4$ the longest branch is not unique, but still the $2j$-th component, $j = 1$, is well-defined.)
\end{setting}

\begin{lem} \label{lem:basis}
	Suppose $A$, $j$, $I_j$ are as above.
	Fix an isomorphism $A \cong k[[x,y,z]] / (f)$ 
	with $f$ is as in Table \ref{table:non-taut RDP}.
	Then, \begin{enumerate}
		\item \label{item:I_j}
		We have $I_j = (x, y^j, z)$,
		\item \label{item:alpha 1}
		The class $[\varepsilon]$ of $\varepsilon := x^{-1} y^{-j} z$ is a generator of the $A$-module $\localcohtorsion{\idealm_A}{A}{I_j}$ with $\Ann([\varepsilon]) = I_j$.
	\end{enumerate}
\end{lem}
\begin{proof}
	(\ref{item:I_j}) 
	Straightforward.
	
	(\ref{item:alpha 1})
	Since $A$ is Gorenstein, we have 
	$\dim_k \localcohtorsion{\idealm_A}{A}{I_j} = \dim_k \Ext^2(A/I_j, A) = \dim_k A/I_j = j$.
	Hence it suffices to check that, 
	in $\Coker (A[x^{-1}] \oplus A[y^{-1}] \to A[(xy)^{-1}])$,
	the classes $[x \varepsilon], [y^j \varepsilon], [z \varepsilon]$ are trivial
	and $[y^{j-1} \varepsilon]$ is nontrivial.
	Straightforward.
\end{proof}

\begin{lem} \label{lem:annihilator}
	Let $A = k[[x,y,z]]/(f)$ be a local ring such that $x,y$ is a regular sequence, and let $j \geq 1$ an integer.
	Let $\varepsilon = x^{-1}y^{-j}z \in A[(xy)^{-1}]$, $I = (x, y^j, z) \subset A$,
	and $e = [(\varepsilon, 0, \dots, 0)] \in \localcoh{\idealm_A}{\Witt[n]{A}}$.
	To show that $e$ belongs to $\localcohtorsion{\idealm_A}{\Witt[n]{A}}{I}$, 
	it suffices to show that $F(R(e)) = 0$ (in $\localcoh{\idealm_A}{\Witt[n-1]{A}}$) and $ze = 0$.
\end{lem}
\begin{proof}
	It suffices to check $(a_0, 0, \dots, 0) \in \Ann(e)$ for $a_0 \in \set{x,y^j,z}$
	and $V(b) \in \Ann(e)$ for $b \in \Witt[n-1]{A}$.
	The former assertion is obvious for $a_0 = x, y^j$ and is assumed for $a_0 = z$.
	For the latter assertion we have $V(b) \cdot e = V(b \cdot F(R(e))) = 0$ since $F(R(e)) = 0$.
\end{proof}

\subsection{Frobenius morphisms} \label{sec:RDP:Frobenius}

Let $A$ be a non-taut RDP.
In this section we 
compute the Frobenius images of certain elements of the local cohomology groups $\localcohtorsion{\idealm_A}{\Witt[n_j]{A}}{I_j}$, 
where $I_j \subset A$ are the $\idealm_A$-primary ideals introduced in Setting \ref{setting:I}.

\begin{prop} \label{prop:local Frob:2:D}
Let $A$ be an RDP of type $D_{N}^r$ in characteristic $p = 2$.
Let $j \geq 1$ and $n \geq 1$ be integers and assume 
\begin{itemize}
\item $\floor{N/2} \geq C_1(n,j) 
                 := 2 j + (2^{n-1} - 1) (2 j - 1) = 2^{n-1} (2 j - 1) + 1$, and
\item either $n = 1$ or $\floor{N/2} - r \geq C_1(n-1, j)$.
\end{itemize}
Let $I_j$ be the ideal defined as in Setting \ref{setting:I}.
Then there is an element $e \in \localcohtorsion{\idealm_A}{\Witt[n]{A}}{I_j}$ satisfying the following two conditions:
\begin{itemize}
	\item
	its restriction $R^{n-1}(e) \in \localcohtorsion{\idealm_A}{A}{I_j}$ is a generator, and
	\item
	its image $F(e)$ by the Frobenius map
	\[
	\map{F}{\localcohtorsion{\idealm_A}{\Witt[n]{A}}{I_j}}{\localcohtorsion{\idealm_A}{\Witt[n]{A}}{\pthpower{I_j}}}
	\]
	satisfies,
	letting $a := \floor{N/2} - r - C_1(n,j)$,
	\[
	F(e) = \begin{cases}
	0 & \text{(if $a \geq 0$)}, \\
	V^{n-1}( e') & \text{(if $a < 0$)}
	\end{cases}
	\]
	for some generator $e' \in {\localcohtorsion{\idealm_A}{A}{I_{-a}}}$.
\end{itemize}
\end{prop}

We will use this proposition (in the proof of Theorem \ref{thm:main}) 
only in the following cases.
\begin{itemize}
\item $\floor{N/2} - r > 2^{n-1}$ and $j = 1$. In this case $a \geq 0$.
\item $\floor{N/2} - r = 2^{n-1} (2 j - 1)$ and $r > 0$. In this case $a = -1$.
\end{itemize}

\begin{proof}
Let $\varepsilon = x^{-1} y^{-j} z$ (with respect to the equation specified below),
and consider the class $e = [(\varepsilon, 0, \dots, 0)] \in \localcoh{\idealm_A}{\Witt[n]{A}}$.
To show $e \in \localcohtorsion{\idealm_A}{\Witt[n]{A}}{I_j}$, it suffices (by Lemma \ref{lem:annihilator}) to check that $ze = 0$ and $F(R(e)) = 0$,
and both follow from the computation of $F(e)$ below (using $z e = x y^j F(e)$).

We will discuss the two cases $N = 2m$ and $N = 2m+1$ ($D_{2m}^r$ and $D_{2m+1}^r$) in a parallel way.
We may assume $A \cong k[[x,y,z]] / (f)$, 
\[
f = \begin{cases}
z^2 + x^2 y + x y^m + z x y^{m-r} & (D_{2m}^r), \\
z^2 + x^2 y + z y^m + z x y^{m-r} & (D_{2m+1}^r),
\end{cases}
\]
and then we have $I_j = (x, y^j, z)$.
Let 
\begin{gather*}
\lambda := y^{m - r - j}, \;
\varepsilon := \frac{z}{x y^j}, \;
\eta := \frac{1}{y^{2j - 1}}, 
\; \xi := \begin{cases}
\dfrac{y^{m - 2j}}{x}    & (D_{2m}^r), \\
\dfrac{y^{m - 2j}z}{x^2} & (D_{2m+1}^r), 
\end{cases}
\end{gather*}
in $A[(xy)^{-1}]$.
Then we have $\varepsilon^2 + \eta + \xi + \lambda \varepsilon = 0$.

As in Lemma \ref{lem:subtraction}(\ref{item:subtraction:2:n}), define polynomials $S_i \in k[t_1,t_2]$
by \[ (t_1 + t_2, 0, \dots) - (t_2, 0, \dots) = (S_0(t_1, t_2), S_1(t_1, t_2), \dots),\]
and let $Q_i := S_i(\xi + \lambda \varepsilon, \eta)$ ($0 \leq i \leq n-1$).
We claim that 
\[
Q_i \equiv 
\begin{cases}
0 & \text{(if $i < n-1$),} \\
\eta^{2^{n-1}-1} \lambda \varepsilon & \text{(if $i = n-1$)}
\end{cases}
\pmod{A[x^{-1}]}.
\]
By Lemma \ref{lem:subtraction}(\ref{item:subtraction:2:n}), 
$Q_i$ is a linear combination of monomials $\xi^{i_1} (\lambda \varepsilon)^{i_2} \eta^{i_3}$
with $i_1 + i_2 + i_3 = 2^{i}$ and $(i_1, i_2, i_3) \neq (0, 0, 2^{i})$.
Let $c(i_1, i_2, i_3) := (m-2j) i_1 + (m-r-2j) i_2 + (-(2j-1)) i_3$,
so that $\xi^{i_1} (\lambda \varepsilon)^{i_2} \eta^{i_3} \in y^{c(i_1, i_2, i_3)} A[x^{-1}]$.
We shall show that $c(i_1, i_2, i_3) \geq 0$ for all such $(i_1, i_2, i_3)$ except $(0, 1, 2^{n-1}-1)$.
\begin{itemize}
\item 
If $i_1 \geq 1$, 
then 
\begin{align*}
c(i_1, i_2, i_3) 
&\geq c(1, 0, 2^i - 1) \\
&= m - 2j - (2^{i} - 1)(2 j - 1) \\
&= m - C_1(i+1,j) \geq m - C_1(n,j),
\end{align*}
which is $\geq 0$ by assumption.
\item 
If $i_2 \geq 1$ and $i < n-1$ (hence $n \geq 2$), then 
\begin{align*}
c(i_1, i_2, i_3) 
&\geq c(0, 1, 2^i - 1) \\
&= m - r - j - (2^{i}-1)(2j-1) - j \\
&= m - r - C_1(i+1, j) \geq m - r - C_1(n-1, j), 
\end{align*}
which is $\geq 0$ by assumption.
\item 
If $i_2 \geq 2$ and $i = n-1 \geq 1$ (hence $n \geq 2$),
then 
\begin{align*}
c(i_1, i_2, i_3) 
&\geq c(0, 2, 2^{n-1} - 2)
 = 2 c(0, 1, 2^{n-2} - 1), 
\end{align*}
which is $\geq 0$ by the previous case.
\end{itemize}

For the remaining term $\eta^{2^{n-1} - 1} {\lambda} \varepsilon$, 
which appears in $Q_{n-1}$ with coefficient $1$ by Lemma \ref{lem:subtraction}(\ref{item:subtraction:2:n}),
we have
\[
\eta^{2^{n-1} - 1} {\lambda} 
= y^{m - r - j - (2^{n-1} - 1)(2 j - 1)}
= y^{m - r - C_1(n,j) + j} 
= y^{a + j},
\]
where $a$ as in the statement.
Hence $\eta^{2^{n-1} - 1} {\lambda} \varepsilon = x^{-1} y^a z$. 
Therefore we have
\begin{align*}
F(\varepsilon, 0, \dots, 0) 
&= (\varepsilon^2, 0, \dots, 0) \\
&= (\xi + \lambda \varepsilon + \eta, 0, \dots, 0) \\
&\equiv (\xi + \lambda \varepsilon + \eta, 0, \dots, 0) - (\eta, 0, \dots, 0) \pmod{\Witt[n]{A[y^{-1}]}} \\
&= (Q_0, \dots, Q_{n-1}) \\
&\equiv (0, \dots, 0, x^{-1} y^a z) \pmod{\Witt[n]{A[x^{-1}]}}.
\end{align*}
If $a \geq 0$ then $x^{-1} y^{a} z \in A[x^{-1}]$, 
and if $a < 0$ then $x^{-1} y^{a} z$ is a generator of $\localcohtorsion{\idealm_A}{A}{I_{-a}}$.
\end{proof}

\begin{prop} \label{prop:local Frob:2:E8}
Let $A$ be an RDP of type $E_8^1$ (resp.\ $E_8^0$) in characteristic $p = 2$.
Then there exists an element $e \in \localcohtorsion{\idealm_A}{A}{I_2}$,
$e \notin \localcohtorsion{\idealm_A}{A}{\idealm_A}$,
where $I_2$ is defined as in Setting \ref{setting:I},
such that $F(e)$ is a generator of $\localcohtorsion{\idealm_A}{A}{\idealm_A}$ (resp.\ $F(e) = 0$).
\end{prop}

\begin{proof}
We may assume $A = k[[x,y,z]] / (z^2 + x^3 + y^5 + b z x y^3)$,
where $b = 1$ (resp.\ $b = 0$).
Then $I_2 = (x, y^2, z)$.
Let $e = [\varepsilon]$ with $\varepsilon := x^{-1} y^{-2} z$.
Since $\varepsilon^2 = y^{-4} x + x^{-2} y + b y \varepsilon 
\equiv b y \varepsilon \pmod{A[x^{-1}] + A[y^{-1}]}$,
we have $F(e) = b [y \varepsilon]$, which is a generator of $\localcohtorsion{\idealm_A}{A}{\idealm_A}$ (resp.\ $0$).
\end{proof}

\begin{prop} \label{prop:local Frob:E:bis}
	Suppose a prime $p$, a Dynkin diagram $S$, and a positive integer $n$ satisfy one of the following.
\begin{enumerate}
\item \label{case:2:E7E8} $p = 2$, $S = E_7,E_8$, $n \leq 3$.
\item \label{case:3:E8} $p = 3$, $S = E_8$, $n \leq 2$.
\item \label{case:2:E6} $p = 2$, $S = E_6$, $n = 1$.
\item \label{case:3:E6E7} $p = 3$, $S = E_6,E_7$, $n = 1$.
\item \label{case:5:E8} $p = 5$, $S = E_8$, $n = 1$.
\end{enumerate}
Let $A$ be an RDP in characteristic $p$ of type $S^r$,
$0 \leq r \leq r_{\max}(p,S) + 1 - n$.
Then there is an element $e \in \localcohtorsion{\idealm_A}{\Witt[n]{A}}{\idealm_A}$ whose restriction 
$R^{n-1}(e)$ is a generator of $\localcohtorsion{\idealm_A}{A}{\idealm_A}$
and satisfying 
	\[
	F(e) = \begin{cases}
	0 & \text{(if $r < r_{\max}(p,S) + 1 - n$)}, \\
	V^{n-1}( e') & \text{(if $r = r_{\max}(p,S) + 1 - n$)}
	\end{cases}
	\]
	for some generator $e' \in \localcohtorsion{\idealm_A}{A}{\idealm_A}$.
\end{prop}

\begin{proof}
In each case we consider the class $e = [(\varepsilon, 0, \dots, 0)]$, $\varepsilon = x^{-1} y^{-1} z$.
To show $e \in \localcohtorsion{\idealm_A}{\Witt[n]{A}}{\idealm_A}$, it suffices (by Lemma \ref{lem:annihilator}) to check that $ze = 0$ and $F(R(e)) = 0$.
The former is clear if $n = 1$ (since $z^2 \in (x,y)$), easy if $p = 2$ using the computation of $F(e)$ (since $z e = x y F(e)$), 
and straightforward in the remaining case ($p = 3$, $S = E_8$, $n = 2$).
The latter follows from the computation of $F(e)$.

\medskip

Case (\ref{case:2:E7E8}): $E_7^r$ (resp.\ $E_8^r$) in characteristic $2$.
	We may assume $A = k[[x,y,z]] / (f)$, 
	$f = z^2 + x^3 + x y^3 + \beta$ 
	(resp.\ $f = z^2 + x^3 + y^5 + \beta$),
	with $\beta$ as in Convention \ref{conv:equations of non-taut RDP} (Table \ref{table:non-taut RDP}).
	Let 
	\[
		\varepsilon := \frac{z}{xy}, \;
		\xi := \frac{y}{x} \; 
		\text{(resp. } \xi := \frac{y^3}{x^2} \text{)}, \;
		\eta := \frac{x}{y^2}, \;
		\omega := \frac{\beta}{x^2 y^2}, \\
	\]
	so we have $\varepsilon^2 + \eta + \xi + \omega = 0$.
	We have 
	\begin{align*}
	\omega &= \varepsilon, \frac{y z}{x^2}, \frac{z}{y}, 0 
	\quad \text{if $A$ is } E_7^r, \; r = 3,2,1,0 \\
	\text{(resp. }
	\omega &= \varepsilon, \frac{y z}{x^2}, \frac{z}{x}, \frac{y z}{x}, 0 
	\quad \text{if $A$ is } E_8^r, r = 4,3,2,1,0 
	\text{)}.
	\end{align*}
	
	Suppose $n = 1$.
	Then $F(\varepsilon) = \varepsilon^2 = \xi + \eta + \omega$,
	all monomials of which belong to $A[x^{-1}] \cup A[y^{-1}]$
	except precisely for $\omega$ in the case $r = r_{\max}$,
	which is equal to $\varepsilon$.
	
	Suppose $n = 2$ and $r \leq r_{\max} - 1$.
	We compute 
	\begin{align*}
	F(\varepsilon, 0) 
	&= (\varepsilon^2, 0) = (\xi + \eta + \omega, 0) \\
	&\equiv (\xi + \eta + \omega) - (\xi, 0) - (\eta, 0) \pmod{\Witt[2]{A[x^{-1}]} + \Witt[2]{A[y^{-1}]}} \\
	&= (\omega, \xi \eta + \xi \omega + \eta \omega).
	\end{align*}
	Then the $0$-th component belongs to $A[x^{-1}]$ or $A[y^{-1}]$, 
	and all monomials in the $1$-st component belong to $A[x^{-1}] \cup A[y^{-1}]$
	except precisely for $\eta \omega$ in the case $r = r_{\max} - 1$,
	which is equal to $\varepsilon$.

	Suppose $n = 3$ and $r \leq r_{\max} - 2$.
	We compute, by using Lemma \ref{lem:subtraction}(\ref{item:subtraction:2:3}) 
	with $(a,b,c) = (\eta, \xi, \omega)$ (resp.\ $(a,b,c) = (\xi, \eta, \omega)$),
	\begin{align*}
	F(\varepsilon, 0, 0)
	&= (\varepsilon^2, 0, 0) = (\eta + \xi + \omega, 0, 0) \\
	&\equiv (\eta + \xi + \omega, 0, 0) - (\xi, \xi \omega, 0)
	\pmod{\Witt[3]{A[x^{-1}]}} \\
	&= (\eta + \omega, \eta \xi, 
	(\eta + \omega)^3 \xi + (\eta + \omega) \xi^3 + (\eta^2 + \eta \omega + \omega^2) \xi^2)
	\\ \text{(resp.\ }
	F(\varepsilon, 0, 0)
	&= (\varepsilon^2, 0, 0) = (\xi + \eta + \omega, 0, 0) \\
	&\equiv (\xi + \eta + \omega, 0, 0) - (\eta, \eta \omega, 0)
	\pmod{\Witt[3]{A[y^{-1}]}} \\
	&= (\xi + \omega, \xi \eta, 
	(\xi + \omega)^3 \eta + (\xi + \omega) \eta^3 + (\xi^2 + \xi \omega + \omega^2) \eta^2)
	\text{).}
	\end{align*}
	Then the $0$-th and $1$-st component belong to $A[y^{-1}]$ (resp.\ $A[x^{-1}]$), 
	and all monomials in the expansion of the $2$-nd component belong to $A[x^{-1}] \cup A[y^{-1}]$
	except precisely for $\eta \omega \xi^2$ (resp.\ $\xi \omega \eta^2$)
	in the case $r = r_{\max} - 2$,
	which is equal to $\varepsilon$.

\medskip

Case (\ref{case:3:E8}): $E_8^r$ in characteristic $3$. 
	We may assume $A = k[[x,y,z]] / (f)$, 
	$f = - z^2 + x^3 + y^5 + \lambda x^2 y^2$,
	where $\lambda = 1, y, 0$ for $r = 2,1,0$ respectively.
	Let $\varepsilon := \frac{z}{xy}$, $\eta := \frac{x}{y^2}$, $\xi := \frac{y^3}{x^2}$,
	so we have $\varepsilon^2 = \xi + \eta + \lambda$.
	
	Suppose $n = 1$.
	Then $F(\varepsilon) = \varepsilon^3 = \xi \varepsilon + \eta \varepsilon + \lambda \varepsilon$,
	all monomials of which belong to $A[x^{-1}] \cup A[y^{-1}]$
	except precisely for $\lambda \varepsilon$ in the case $r = r_{\max}$,
	which is equal to $\varepsilon$.
	
	Suppose $n = 2$ and $r \leq r_{\max} - 1 = 1$.
	We compute, by using Lemma \ref{lem:subtraction}(\ref{item:subtraction:3:2}), 
	\begin{align*}
	F(\varepsilon, 0) 
	&= (\varepsilon^3, 0) = ((\xi + \lambda) \varepsilon + \eta \varepsilon, 0) \\
	&\equiv ((\xi + \lambda) \varepsilon + \eta \varepsilon, 0) - ((\xi + \lambda) \varepsilon, 0) - (\eta \varepsilon, 0)
	 \pmod{\Witt[2]{A[x^{-1}]} + \Witt[2]{A[y^{-1}]}} \\
	&= (0, \eta \varepsilon \cdot (\xi + \lambda) \varepsilon \cdot (\xi + \eta + \lambda) \varepsilon) \\
	&= (0, \varepsilon \eta (\xi + \lambda) (\xi + \eta + \lambda)^2).
	\end{align*}
	Write $\lambda = b y$, where $b = 1,0$ for $r = 1,0$.
	For the $1$-st component, we have
	\begin{align*}
		\varepsilon \eta (\xi + \lambda) (\eta + \xi + \lambda)^2
		&= \frac{z}{x^7 y^7} (x \cdot (y^3 + b x^2 y) \cdot (x^3 + y^5 + b x^2 y^3)^2) \\
		&= \frac{z}{x^6 y^6} (y^2 + b x^2) (2 x^3 y^5 + 2 b x^5 y^3 + \negligible{x^6 + y^{10} + b^2 x^4 y^6 + 2 b x^2 y^8} ) \\
		&= \frac{z}{x^6 y^6} (4 b x^5 y^5 + \negligible{2 x^3 y^7 + 2 b^2 x^7 y^3} + (y^2 + b x^2) \negligible{x^6 + y^{10} + b^2 x^4 y^6 + 2 b x^2 y^8}) \\
		&\equiv b \varepsilon \pmod{A[x^{-1}] + A[y^{-1}]},
	\end{align*}
	where $\negligible{\dots}$ indicates that the terms inside belong to $(x^6, y^6) \subset k[[x,y]]$ and are thus negligible.
	
\medskip

The remaining cases:
We may assume $A = k[[x,y,z]] / (f)$ with 
\[
f = \begin{cases}
 z^2 + x^3   +   y^2 z + b x y z     & ((p, S) = (2, E_6)), \\
-z^2 + x^3   +   y^4   + b x^2 y^2   & ((p, S) = (3, E_6)), \\
-z^2 + x^3   + x y^3   + b x^2 y^2   & ((p, S) = (3, E_7)), \\
 z^2 + x^3   +   y^5   + (b/2) x y^4 & ((p, S) = (5, E_8)), 
\end{cases}
\]
for some $b \in k$ with $b = 0$ if $r = 0$ and $b \neq 0$ if $r = 1$.
Let $\varepsilon := x^{-1} y^{-1} z$.
It suffices to show that $\varepsilon^p - b \varepsilon = \eta + \xi$ 
for some $\eta \in A[y^{-1}]$ and $\xi \in A[x^{-1}]$.
We take 
\[
\eta := \begin{cases}
y^{-2} x, \\
y^{-3} z, \\
y^{-3} z, \\
y^{-5} x z,
\end{cases}
\quad
\xi := \begin{cases}
x^{-2} z, \\
x^{-3} y z, \\
x^{-2} z, \\
x^{-5} (y^5 + b x y^4 + (b^2/4) x^2 y^3 + 2 x^3)z.
\end{cases}
\]
\end{proof}
\begin{rem}
A consequence of Proposition \ref{prop:local Frob:2:D} ($n = j = 1$) and Proposition \ref{prop:local Frob:E:bis} ($n = 1$)
is that a non-taut RDP is $F$-injective (i.e.\ the Frobenius action on $H^2_{\idealm_A}(A)$ is injective) if and only if $r = r_{\max}$.
Tanaka \cite{Tanaka:taut}*{Section 5.2}, using Fedder's criterion for $F$-purity, observed that a non-taut RDP is $F$-pure if and only if $r = r_{\max}$.
Note that, for Gorenstein singularities, $F$-injectivity is equivalent to $F$-purity \cite{Takagi--Watanabe:F-singularities}. 
\end{rem}

Finally we note the following relation between RDPs connected by partial resolutions
(although we do not need it in this paper).
If $Z$ is an RDP surface with an RDP $z$ of type $S$ and $S' \subset S$ is a subdiagram,
then the minimal resolution $\map{\rho}{\tilde{Z}}{Z}$ of $Z$ at $z$ 
factors through the contraction $\map{\rho'}{\tilde{Z}}{Z_{S'}}$ of $S' \subset S = \Exc(\rho)$, 
and $Z_{S'}$ is an RDP surface.
We say that $Z_{S'} \to Z$ is the \emph{partial resolution} corresponding to $S' \subset S$.
If $S'$ is connected and non-empty,
then $Z_{S'}$ has a single RDP above $z$, which is of type $S'$.

\begin{conv} \label{conv:rmax=0}
	We abuse the notation and say that an RDP of type $S$ is of type $S^0$ if $r_{\max}(p,S) = 0$,
	so that the coindex $r$ of an RDP is always defined.
\end{conv}

\begin{lem} \label{lem:partial resolution}
	Let $S' \subset S$ be a non-empty connected subdiagram of a Dynkin diagram $S$.
	Let $Z$ be an RDP of type $S^r$
	and $Z_{S'} \to Z$ be the partial resolution corresponding to $S' \subset S$.
	Suppose $z'$ is of type $S'^{r'}$.
	Then we have $r' = \max\set{0, r - (r_{\max}(S) - r_{\max}(S'))}$.
\end{lem}
In other words, we have $r_{\max}(S') - r' = r_{\max}(S) - r$ if this equality is achieved by a non-negative integer $r'$, and $r' = 0$ otherwise.

\begin{proof}
	We may assume that the number of components of $S'$ is one less than that of $S$.
	If $r_{\max}(S') = 0$ then the assertion is trivial.
	So we may assume $r_{\max}(S') > 0$.
	
	If $(S,S')$ is $(E_8,E_7)$ or $(E_7,D_6)$, then the partial resolution is the blow-up at the closed point.
	In the other cases, the partial resolution is the blow-up at the ideals
	$(x,y^2,z)$, $(y,z)$, or $(x,z)$, as displayed in Table \ref{table:partial resolution},
	with respect to the equations given in Table \ref{table:non-taut RDP}. 
	One can check that this blow-up is dominated by the thrice blow-up $Z_3$,
	where $Z_0 := Z$ and $Z_{i+1} := \Bl_{\Sing(Z_i)} Z_i$, 
	hence it is indeed a partial resolution.
	A straightforward computation proves the assertion in each case.
	\begin{table}
		\caption{Partial resolutions of RDPs} \label{table:partial resolution}
		\begin{tabular}{lllllll}
			\toprule
			$p$   & $S$ & $S'$ & $r_{\max}(S)$ & $r_{\max}(S')$ & equation of $S^r$ & add: \\
			\midrule
			$2$   & $D_{2m}$   & $D_{2m-1}$ & $m-1$ & $m-2$ & $z^2 + x^2 y + x y^m + \dots$ & $z/y$ \\
			$2$   & $D_{2m+1}$ & $D_{2m}$   & $m-1$ & $m-1$ & $z^2 + x^2 y + z y^m + \dots$ & $z/y$ \\
			$2$   & $E_8$      & $E_7$      & $4$   & $3$   & $z^2 + x^3 + y^5     + \dots$ & $x/y, z/y$ \\
			$2$   & $E_8$      & $D_7$      & $4$   & $2$   & $z^2 + x^3 + y^5     + \dots$ & $z/x, y^2/x$ \\
			$2$   & $E_7$      & $D_6$      & $3$   & $2$   & $z^2 + x^3 + x y^3   + \dots$ & $x/y, z/y$ \\
			$2$   & $E_7$      & $E_6$      & $3$   & $1$   & $z^2 + x^3 + x y^3   + \dots$ & $z/x$ \\
			$2$   & $E_6$      & $D_5$      & $1$   & $1$   & $z^2 + x^3 + y^2 z   + \dots$ & $z/x$ \\
			\midrule
			$3$   & $E_8$      & $E_7$      & $2$   & $1$   & $-z^2 + x^3 + y^5    + \dots$ & $x/y, z/y$ \\
			$3$   & $E_7$      & $E_6$      & $1$   & $1$   & $-z^2 + x^3 + x y^3  + \dots$ & $z/x$ \\
			\bottomrule
		\end{tabular}
	\end{table}
\end{proof}

\subsection{\texorpdfstring{$\mu_p$- and $\alpha_p$-quotient morphisms}{mu\_p- and alpha\_p-quotient morphisms}} \label{sec:RDP:mu alpha}

\begin{prop} \label{prop:local mu_p alpha_p}
	Suppose a prime $p$, a group scheme $G$, a Dynkin diagram $S$, and a positive integer $n$ satisfy one of the following.
\begin{enumerate}
\item \label{case:p:mu} $p$ is arbitrary, $G = \mu_p$, $S = A_{p-1}$, $n = 1$.
\item \label{case:2:alpha:D4D8} $p = 2$, $G = \alpha_p$, $S = D_{2^n}$, $n \geq 2$.
\item \label{case:2:alpha:E8} $p = 2$, $G = \alpha_p$, $S = E_8$, $n = 4$.
\item \label{case:3:alpha:E6} $p = 3$, $G = \alpha_p$, $S = E_6$, $n = 2$.
\item \label{case:5:alpha:E8} $p = 5$, $G = \alpha_p$, $S = E_8$, $n = 2$.
\end{enumerate}
Let $\map{\pi}{\Spec B}{\Spec A}$ be a $G$-quotient map 
from a smooth point $B$ to an RDP $A$ of type $S^0$ in characteristic $p$,
with $\Fix(G) = \set{\idealm_B}$. 
Then there is an element $e \in \localcohtorsion{\idealm_A}{\Witt[n]{A}}{\idealm_A}$ whose restriction 
$R^{n-1}(e)$ is a generator of $\localcohtorsion{\idealm_A}{A}{\idealm_A}$
and satisfying $\pi^*(e) = V^{n-1}(e')$
for a generator $e' \in \localcohtorsion{\idealm_B}{B}{\idealm_B}$.
\end{prop}

\begin{proof}

In each case the assumptions determine $\Spec B \to \Spec A$ up to isomorphism by \cite{Matsumoto:k3alphap}*{Theorem 3.3(1)}.
In each case we consider the class of the form $e = [(\varepsilon, 0, \dots, 0)]$, $\varepsilon = x^{-1} y^{-1} z^j$.
We can check $e \in \localcohtorsion{\idealm_A}{\Witt[n]{A}}{\idealm_A}$
as in the beginning of the proof of Proposition \ref{prop:local Frob:E:bis}.

\medskip
	
Case (\ref{case:p:mu}): $A_{p-1}$ in characteristic $p$. 
We may assume $B = k[[X,Y]]$ and $A = k[[x,y,z]] / (z^p - x y)$
with $x = X^p$, $y = Y^p$, $z = X Y$.
Then it is clear that $[x^{-1} y^{-1} z^{p-1}]$ and 
$\pi^*([x^{-1} y^{-1} z^{p-1}]) = [X^{-1} Y^{-1}]$
are generators of $\localcohtorsion{\idealm_A}{A}{\idealm_A}$ and $\localcohtorsion{\idealm_B}{B}{\idealm_B}$ respectively.

\medskip

Case (\ref{case:2:alpha:D4D8}): $D_{2^n}^0$ in characteristic $2$. 
We may assume $B = k[[X,Y]]$ and $A = k[[x,y,z]] / (z^2 + x^2 y + x y^{2^{n-1}})$
with $x = X^2$, $y = Y^2$, $z = X^2 Y + X Y^{2^{n-1}}$.
Let $\varepsilon = x^{-1} y^{-1} z$.
We compute, by using Lemma \ref{lem:subtraction}(\ref{item:subtraction:2:n})
as in the proof of Proposition \ref{prop:local Frob:2:D},
\begin{align*}
\pi^*(\varepsilon, 0, \dots, 0) 
&= \Bigl( \frac{X^2 Y + X Y^{2^{n-1}}}{X^2 Y^2}, 0, \dots, 0 \Bigr)
 = \Bigl( \frac{1}{Y} + \frac{Y^{2^{n-1} - 2}}{X}, 0, \dots, 0 \Bigr) \\
&\equiv \Bigl( \frac{1}{Y} + \frac{Y^{2^{n-1} - 2}}{X}, 0, \dots, 0 \Bigr) - \Bigl( \frac{1}{Y}, 0, \dots, 0 \Bigr) \\
&= \Bigl( \xi_0, \dots, \xi_{n-2}, \xi_{n-1} + \frac{Y^{2^{n-1} - 2}}{X} \Bigl( \frac{1}{Y} \Bigr)^{2^{n-1} - 1} \Bigr) \\
&\equiv V^{n-1} \Bigl( \frac{1}{XY} \Bigr) \pmod{\Witt[n]{B[x^{-1}]} + \Witt[n]{B[y^{-1}]}}, 
\end{align*}
for some $\xi_i \in B[x^{-1}]$.
Clearly $[X^{-1} Y^{-1}]$ is a generator of $\localcohtorsion{\idealm_B}{B}{\idealm_B}$.

\medskip

Case (\ref{case:2:alpha:E8}): $E_8^0$ in characteristic $2$. 
We may assume $B = k[[X,Y]]$ and $A = k[[x,y,z]] / (z^2 + x^3 + y^5)$
with $x = X^2$, $y = Y^2$, $z = X^3 + Y^5$.
Let $\varepsilon = x^{-1} y^{-1} z$.
We compute, by using Lemma \ref{lem:subtraction}(\ref{item:subtraction:2:4}),
\begin{align*}
\pi^*(\varepsilon, 0, 0, 0) 
&= \Bigl( \frac{X^3 + Y^5}{X^2 Y^2}, 0, 0, 0 \Bigr)
 = \Bigl( \frac{X}{Y^2} + \frac{Y^3}{X^2} , 0, 0, 0 \Bigr) \\
&\equiv \Bigl( \frac{X}{Y^2} + \frac{Y^3}{X^2} , 0, 0, 0 \Bigr) 
- \Bigl( \frac{X}{Y^2}, 0, 0, 0 \Bigr) 
- \Bigl( \frac{Y^3}{X^2} , 0, 0, 0 \Bigr) \\
&= \Bigl(0, \frac{Y}{X}, \frac{Y^7 + Y^2 X^3}{X^5} + \frac{X}{Y^3}, 
\frac{Y^{19} + Y^{14} X^3 + Y^4 X^9}{X^{13}} + \frac{X^2 Y^5 + X^5}{Y^{11}} \Bigr) \\
&= \Bigl(0, \xi_1, \frac{Y^7 + Y^2 X^3}{X^5} + \frac{X}{Y^3}, \xi_3 + \eta_3 \Bigr) \\
&\equiv \Bigl(0, \xi_1, \frac{Y^7 + Y^2 X^3}{X^5} + \frac{X}{Y^3}, \xi_3 + \eta_3 \Bigr)
 - \Bigl( 0, 0, \frac{X}{Y^3}, 0 \Bigr) \\
&= \Bigl(0, \xi_1, \xi_2, \xi_3 + \eta_3 + \frac{Y^7 + Y^2 X^3}{X^5} \frac{X}{Y^3} \Bigr) \\
&\equiv V^3 \Bigl( \frac{1}{X Y} \Bigr) \pmod{\Witt[4]{B[x^{-1}]} + \Witt[4]{B[y^{-1}]}}, 
\end{align*}
where $\xi_i \in B[x^{-1}]$ and $\eta_i \in B[y^{-1}]$.

\medskip

Case (\ref{case:3:alpha:E6}): $E_6^0$ in characteristic $3$. 
We may assume $B = k[[Y,Z]]$ and $A = k[[x,y,z]] / (-z^2 + x^3 + y^4)$
with $x = Z^2 - Y^4$, $y = Y^3$, $z = Z^3$.
We interpret the local cohomology groups using the regular sequence $x,y$. 
Let $\varepsilon = x^{-1} y^{-1} z$.
We compute, by using Lemma \ref{lem:subtraction}(\ref{item:subtraction:3:2})
and the equality $Z^2 = x + Y^4$,
\begin{align*}
\pi^*(\varepsilon, 0) 
&= \Bigl(\frac{Z^3}{x Y^3}, 0\Bigr) 
 = \Bigl(\frac{Z}{Y^3} + \frac{ZY}{x}, 0\Bigr) \\
&\equiv \Bigl(\frac{Z}{Y^3} + \frac{ZY}{x}, 0\Bigr) - \Bigl(\frac{Z}{Y^3}, 0\Bigr) - \Bigl(\frac{ZY}{x}, 0\Bigr) \\
&= \Bigl(0, \frac{Z}{Y^3} \frac{ZY}{x} \frac{Z^3}{x Y^3} \Bigr) 
 = \Bigl(0, \frac{Z (x + Y^4)^2}{x^2 Y^5} \Bigr) 
 = \Bigl(0, \frac{Z \cdot (-x Y^4 + \negligible{x^2 Z + Y^8 Z})}{x^2 Y^5} \Bigr) \\
&\equiv \Bigl(0, -\frac{Z}{x Y} \Bigr) \pmod{\Witt[2]{B[x^{-1}]} + \Witt[2]{B[y^{-1}]}} \\
&= V(\beta),
\end{align*} 
where $\negligible{\dots} \in (x^2, Y^8)$ are negligible
and $\beta := -x^{-1}Y^{-1}Z$.
Then $e' := [\beta]$ is a generator of $\localcohtorsion{\idealm_B}{B}{\idealm_B}$ since $\Ann([\beta]) = (Y,Z) = \idealm_B$.

\medskip

Case (\ref{case:5:alpha:E8}): $E_8^0$ in characteristic $5$. 
We may assume $B = k[[X,Z]]$ and $A = k[[x,y,z]] / (z^2 + x^3 - y^5)$
with $x = X^5$, $y = Z^2 + X^3$, $z = Z^5$.
We interpret the local cohomology groups using the regular sequence $x,y$. 
Let $\varepsilon = x^{-1} y^{-1} z$.
We compute, by using Lemma \ref{lem:subtraction}(\ref{item:subtraction:5:2})
and the equality $Z^2 = y - X^3$,
\begin{align*}
\pi^*(\varepsilon, 0) 
&= \Bigl(\frac{Z^5}{X^5y}, 0\Bigr) 
 	= \Bigl(\frac{Z (y - 2 X^3)}{X^5} + \frac{ZX}{y}, 0\Bigr) \\
&\equiv \Bigl(\frac{Z (y - 2 X^3)}{X^5} + \frac{ZX}{y}, 0\Bigr) - \Bigl(\frac{Z (y - 2 X^3)}{X^5}, 0\Bigr) - \Bigl(\frac{ZX}{y}, 0\Bigr) \\
&= \Bigl(0, \frac{Z^5}{X^5 y} \frac{Z (y - 2 X^3)}{X^5} \frac{ZX}{y} \frac{Z^2 (y^2(y-2X^3)^2 + X^6y(y-2X^3) + X^{12})}{X^{10} y^2} \Bigr) \\
&= \dots = \Bigl(0, \frac{Z \cdot (\negligible{ y^9 - 2 X^6 y^7 + X^9 y^6 + X^{12} y^5 - 2 X^{15} y^4 } - X^{18} y^3 + \negligible{X^{21} y^2 - 2 X^{24} y -2 X^{27}})}{X^{19} y^4} \Bigr) \\
&\equiv \Bigl(0, -\frac{Z}{Xy} \Bigr) \pmod{\Witt[2]{B[x^{-1}]} + \Witt[2]{B[y^{-1}]}} \\
&= V(\beta),
\end{align*} 
where $\negligible{\dots} \in (X^{21}, y^4)$ are negligible
and $\beta := -X^{-1}y^{-1}Z$.
Then $e' := [\beta]$ is a generator of $\localcohtorsion{\idealm_B}{B}{\idealm_B}$ since $\Ann([\beta]) = (X,Z) = \idealm_B$.
\end{proof}

\section{The height of K3 surfaces} \label{sec:height}

In this section we recall the definition and properties of the height of K3 surfaces.

Recall that a \emph{K3 surface} over a field $k$ is a smooth proper surface $Y$ 
satisfying $H^1(Y, \cO_Y) = 0$ and $\Omega^2_Y \cong \cO_Y$.
We have $\dim H^2(Y, \cO_Y) = 1$.
We say that a proper surface $Y$ is an \emph{RDP K3 surface}
if it has only rational double points as singularities (if any) and its minimal resolution is a K3 surface.

\begin{lem} \label{lem:Witt of RDP K3}
Let $\map{\pi}{\tilde{Y}}{Y}$ be the minimal resolution of an RDP K3 surface. 
Then $H^1(Y, \Witt[n]{\cO_Y}) = 0$, and
$H^2(Y, \Witt[n]{\cO_Y}) \to H^2(\tilde{Y}, \Witt[n]{\cO_{\tilde{Y}}})$ are isomorphisms for all $n \geq 1$.
\end{lem}
\begin{proof}
	Since $Y$ has only rational singularities, we have $R^i \pi_* \cO_{\tilde{Y}} = 0$ for $i > 0$.
	Hence $H^1(Y, \cO_Y) = H^1(\tilde{Y}, \cO_{\tilde{Y}}) = 0$, and 
	$H^2(Y, \cO_Y) \to H^2(\tilde{Y}, \cO_{\tilde{Y}})$ is an isomorphism.

For each $0 \leq n' \leq n$,
there is an exact sequence 
\[
0 \to \Witt[n-n']{\cO_Y} \namedto{V^{n'}} \Witt[n]{\cO_Y} \namedto{R^{n-n'}} \Witt[n']{\cO_Y} \to 0.
\]
Since $H^1(Y, \cO_Y) = 0$, it follows inductively that $H^1(Y, \Witt[n]{\cO_Y}) = 0$ for all $n \geq 1$.
Hence we obtain an exact sequence
\[
0 \to H^2(Y, \Witt[n-n']{\cO_Y}) \namedto{V^{n'}} H^2(Y, \Witt[n]{\cO_Y}) \namedto{R^{n-n'}} H^2(Y, \Witt[n']{\cO_Y}) \to 0,
\]
and this is compatible with pullbacks by $\pi$.
It follows inductively that $H^2(Y, \Witt[n]{\cO_Y}) \to H^2(\tilde{Y}, \Witt[n]{\cO_{\tilde{Y}}})$ are isomorphisms for all $n \geq 1$.
\end{proof}

\begin{thm}[Artin--Mazur \cite{Artin--Mazur:formalgroups}*{Corollary II.4.2}]
Let $Y$ be a (smooth) K3 surface.
The functor
$\map{\Phi^2}{\cat{\text{local Artinian $k$-algebras}}}{\cat{\text{abelian groups}}}$
defined by $S \mapsto \Ker (\Het^2(Y \times S, \Gm) \to \Het^2(Y, \Gm))$
is pro-represented by a $1$-dimensional formal group.
\end{thm}
\begin{defn}
This formal group is called the \emph{formal Brauer group} of $Y$ and written $\widehat{\Br}(Y)$.
If $\charac k = p > 0$, then
its height is called the \emph{(Artin--Mazur) height} of $Y$ and written $\height(Y)$.

We define the height of an RDP K3 surface
to be the height of its minimal resolution.
\end{defn}
Here a $1$-dimensional commutative formal group in characteristic $p > 0$ is said to be of height $h \in \bZ_{> 0}$ 
if $[p](t) = c t^{p^h} + \dots$ for some $c \in k^*$,
and of height $\infty$ if $[p](t) = 0$, 
where $t$ is a uniformizer and $[p]$ is the multiplication-by-$p$ map.
It follows from Proposition \ref{prop:bound of height} that $\height(Y) \in \set{1, 2, \dots, 10} \cup \set{\infty}$.

To relate the height of a K3 surface with the properties of non-taut RDPs, we need  
the following characterization of the height in terms of Frobenius actions on $W_n$-valued cohomology.
\begin{thm}[van der Geer--Katsura \cite{vanderGeer--Katsura:stratification}*{Theorem 5.1}] \label{thm:vdGK}
	Let $Y$ be an RDP K3 surface in positive characteristic
	and $n \geq 1$ an integer.
	Then $\height(Y) \leq n$ if and only if 
	the Frobenius map on $H^2(Y, \Witt[n]{\cO_Y})$ is nonzero.
\end{thm}
\begin{proof}
	In \cite{vanderGeer--Katsura:stratification} this is stated for smooth K3 surfaces.
	The case of RDP K3 surfaces is reduced to the smooth case by using Lemma \ref{lem:Witt of RDP K3}.
\end{proof}

We also recall several properties that can be used to determine or bound the height of (RDP) K3 surfaces.

\begin{prop} \label{prop:height:wxyz}
	Suppose $\bP = \bP(n_0, n_1, n_2, n_3) = \Proj k[x_0, x_1, x_2, x_3]$ 
	is a $3$-dimensional weighted projective space, 
	and $Y = (f = 0) \subset \bP$ is a hypersurface of degree $\deg(f) = \sum n_i$ that is an RDP K3 surface.
	Then, $Y$ is ordinary (i.e.\ $\height(Y) = 1$) if and only if 
	the coefficient of $(x_0 x_1 x_2 x_3)^{p-1}$ in $f^{p-1}$ is nonzero.
\end{prop}
\begin{proof}
	The proof is standard and applicable to hypersurface Calabi--Yau varieties of arbitrary dimension, 
	see for example \cite{Hartshorne:AG}*{Proposition IV.4.21} for the $1$-dimensional case.
	We include the proof for the reader's convenience.
	Let $d = \deg(f)$.
	We have canonical isomorphisms
	\begin{align*}
	H^2(X, \cO_X) \isomto H^3(\bP, f \cO_{\bP}(-d)) &\isomfrom \check{H}^3(\set{U_i}_{i \in I}, f \cO_{\bP}(-d)) \\
	&= \Coker \Bigl( \bigoplus_{\card{J} = 3} \Gamma(U_{J}, f \cO(-d)) 
	\to \Gamma(U_{I}, f \cO(-d))  \Bigr), 
	\end{align*}
	where $\set{U_i}_{i \in I}$ ($I := \set{0,1,2,3}$) is the standard affine covering of $\bP$
	and $U_{J} := \bigcap_{i \in J} U_i$ for $J \subset I$.
	This cokernel is $1$-dimensional, generated by the class of $\frac{f}{x_0 x_1 x_2 x_3}$.
	The Frobenius image of this class is 
	the class of $\frac{f^p}{(x_0 x_1 x_2 x_3)^{p}} 
	= f \cdot \frac{f^{p-1}}{(x_0 x_1 x_2 x_3)^{p}} \in \Gamma(U_{I}, f \cO(-d))$,
	which is nontrivial if and only if the coefficient of $(x_0 x_1 x_2 x_3)^{p-1}$ in $f^{p-1}$ is nonzero.
	Apply Theorem \ref{thm:vdGK}.
\end{proof}

\begin{rem}
	In principle, it is possible to compute the Frobenius map on $H^2(X, \Witt[n]{\cO_X})$ in terms of $f$.
\end{rem}

\begin{thm} \label{thm:crystalline slope}
	Let $Y$ be a (smooth) K3 surface.
	Consider the crystalline cohomology group $\Hcrys^2(Y/\Witt{k})$, which is an $F$-crystal.
	If $h = \height(Y) < \infty$,
	then $\Hcrys^2(Y/\Witt{k})$ has slopes $1 - 1/h$, $1$, and $1 + 1/h$,
	with respective multiplicity $h$, $22 - 2 h$, and $h$.
	If $\height(Y) = \infty$,
	then it has slope $1$ with multiplicity $22$.
\end{thm}
\begin{proof}
	By \cite{Artin--Mazur:formalgroups}*{Corollary II.4.3},
	the Dieudonn\'e module of $\widehat{\Br}(Y)$ is isomorphic to $H^2(Y, \WO_Y)$.
	The slope spectral sequence induces an isomorphism
	$H^2(Y, \WO_Y) \otimes_{W(k)} K_0 \cong \Hcrys^2(Y/W(k))_{<1} \otimes_{W(k)} K_0$,
	where $K_0 := \Frac W(k)$
	and $-_{<1}$ denotes the slope $<1$ part of an $F$-crystal.
	The assertion follows from this (see \cite{Illusie:deRham--Witt}*{Section II.7.2}).
\end{proof}

\begin{prop}[\cite{Illusie:deRham--Witt}*{Proposition II.5.12}] \label{prop:bound of height}
	Suppose $Y$ is a (smooth) K3 surface 
	of height $h$ with Picard number $\rho = \rho(Y) := \rank \Pic(Y)$.
	If $h < \infty$, then $\rho \leq 22 - 2 h$,
	and if $h = \infty$, then $\rho \leq 22$.
\end{prop}
\begin{proof}
	The subspace of $\Hcrys^2(Y / \Witt{k})$ generated by the Picard group
	is of slope $1$ with multiplicity $\rho$,
	which should be at most $22 - 2h$ (resp.\ $22$) if $h < \infty$ (resp.\ $h = \infty$)
	by Theorem \ref{thm:crystalline slope}.
\end{proof}
\begin{cor} \label{cor:bound of height}
	Suppose $Y$ is an RDP K3 surface of height $h$ with RDPs $z_i$ of type $A_{N_i}$, $D_{N_i}$, or $E_{N_i}$.
	If $h < \infty$, then $\sum N_i < 22 - 2 h$,
	and if $h = \infty$, then $\sum N_i < 22$.
\end{cor}
\begin{proof}
	The exceptional curves on the minimal resolution $\tilde{Y}$ generate 
	a negative-definite sublattice of $\Pic(\tilde{Y})$ of rank $\sum N_i$.
	Since $\Pic(\tilde{Y})$ is of sign $(+1, -(\rho-1))$, we have $\sum N_i < \rho$.	
\end{proof}

\begin{prop} \label{prop:height:point counting}
	Suppose $Y$ is an RDP K3 surface defined over a finite field $\bF_q$.
	Define $a(m) \in \bQ$ by $\card{Y(\bF_{q^m})} = 1 + (q^m)^2 + a(m) q^m$.
	Then there exist a family $x_1, \dots, x_u \in \overline{\bQ}$ such that $a(m) = \sum_i x_i^m$.
	Let $s(j)$ be the $j$-th elementary symmetric polynomial of $x_1, \dots, x_u$.
	Then $\height(Y) > n$ if and only if $s(1), \dots, s(n) \in (p/q)\bZ$.
\end{prop}
Note that each $s(j)$ can be expressed as a polynomial of $a(1), \dots, a(j)$ with coefficients in $\bQ$,
without knowing the indeterminates $x_i$ or even the number of indeterminates.
First few examples are 
$s(1) = a(1)$, 
$s(2) = (a(1)^2 - a(2)) / 2$, and
$s(3) = (a(1)^3 - 3 a(1) a(2) + 2 a(3)) / 6$.
\begin{proof}

	We first consider the case $Y$ is smooth.
	Write $q = p^b$.
	Then $F^b \in \End(\Hcrys^2(Y/W(\bF_q)))$ is linear (not only semilinear), 
	its characteristic polynomial have coefficients in $\bZ$,
	and this coincides with 
	the characteristic polynomial of the $q$-th power Frobenius on $\Het^2(Y, \bQ_l)$ for any prime $l \neq p$
	(\cite{Illusie:report}*{3.7.3}).
	Let $y_i$ ($1 \leq i \leq 22$) be the eigenvalues, and let $x_i = q^{-1} y_i$. 
	The Weil conjecture (more precisely, the Lefschetz trace formula) asserts that
	$x_i$ satisfy $\card{Y(\bF_{q^m})} = 1 + (q^m)^2 + q^m \sum_{i} x_i^m$ for all $m \geq 1$.	
	Then $a(m) := \sum_i x_i^m$ and $x_i$ satisfy the required properties.
	Let $S(j)$ and $s(j)$ be respectively the $j$-th elementary symmetric polynomials of $y_i$ and of $x_i$
	($S(j)$ are the coefficients of the characteristic polynomial).
	The $p$-adic valuations of $S(j)$ are encoded in the Newton polygon of $\Hcrys^2(Y/W(k))$,
	and the Newton polygon is described (Theorem \ref{thm:crystalline slope})
	in terms of $h = \height(Y)$ as follows.
	Let $\ord_p$ be the $p$-adic valuation.
	\begin{itemize}
	\item If $h < \infty$,
	then $\ord_p(S(j)) \geq j ((h-1)/h) \ord_p(q)$ 
	(equivalently $\ord_p(s(j)) \geq -(j/h) \ord_p(q)$) for all $0 \leq j \leq h$,
	and the equality holds for $j = h$.
	\item If $h = \infty$, 
	then $\ord_p(S(j)) \geq j \ord_p(q)$ (equivalently $\ord_p(s(j)) \geq 0$) for all $j$.
	\end{itemize}
	From this, it follows that 
	$h > n$ if and only if $s(1), \dots, s(n) \in (p/q) \bZ$.
	We also observe that, for any fixed embedding $\overline{\bQ} \injto \overline{\bQ_p}$,
	the number of $x_i$ with negative valuation is equal to $\height(Y)$ if $\height(Y) < \infty$ and to $0$ if $\height(Y) = \infty$.

	Now consider an RDP K3 surface $Y$ and its minimal resolution $\map{\pi}{\tilde{Y}}{Y}$.
	Let $E_1, \dots, E_t \subset \tilde{Y}_{\overline{\bF_q}}$ be the exceptional curves,
	and $x_1, \dots, x_t$ be the eigenvalues of the subspace of 
	$\Het^2(\tilde{Y}_{\overline{\bF_q}}, \bQ_l)$ generated by the Chern classes of $E_i$.
	Then all $x_i$ are roots of unity.
	We will show that $\card{\tilde{Y}(\bF_{q^m})} = \card{Y(\bF_{q^m})} + q^m \sum_{i = 1}^{t} x_i^m$.
	Suppose this for the moment.
	Then we have $\card{Y(\bF_{q^m})} = 1 + q^{2m} + q^m \sum_{i = t+1}^{22} x_i^m$.
	Let $\tilde{s}(j)$ and $s(j)$ be respectively the $j$-th elementary symmetric polynomials 
	of $x_i$ ($1 \leq i \leq 22$) and $x_i$ ($t+1 \leq i \leq 22$).
	Then, since the number of elements with negative valuation are equal for the two families,
	the height predicted by the sequence $s(j)$ coincides with that by the sequence $\tilde{s}(j)$.

	To show the equality on the number of rational points, 
	it suffices to show $\card{Y'(\bF_{q^m})} = \card{Y(\bF_{q^m})} + q^m \sum_{i = 1}^{t'} x_i^m$,
	where $\map{\pi'}{Y'}{Y}$ is the blowup at one singular point $y \in Y$
	and $x_i$ ($1 \leq i \leq t'$) are the eigenvalues corresponding to $\pi'$-exceptional curves.
	Let $k(y) = \bF_{q^r}$. 
	Let $z \in Y_{k(y)}$ be a point above $y$
	and $E_1, \dots, E_s \subset Y'_{k(y)}$ be the exceptional curves above $z$.
	Since $z$ is a double point, the possibilities for $s$ and the components $E_1, \dots, E_s$ are as follows.
	\begin{enumerate}
	\item $s = 1$, $E_1$ is a smooth conic (with an $k(y)$-rational point), \label{case:smooth}
	\item $s = 2$, $E_1$ and $E_2$ are two distinct lines. \label{case:split two lines}
	\item $s = 1$, $E_1$ is a non-smooth conic that splits into two distinct lines over the quadratic extension $\bF_{q^{2r}}$ of $k(y)$. \label{case:non-split two lines}
	\item $s = 1$, $E_1$ is a double line. \label{case:double line}
	\end{enumerate} 
	We have $t' = 2 s r$ in case (\ref{case:non-split two lines}), and $t' = s r$ in all other cases.
	In cases (\ref{case:smooth}), (\ref{case:non-split two lines}), and (\ref{case:double line}), 
	the eigenvalues $x_1, \dots, x_{t'}$ are all $t'$-th roots of unity.
	In case (\ref{case:split two lines}), the eigenvalues are all $r$-th roots of unity, each appearing twice.
	In each case, it is straightforward to check the required equality.
\end{proof}

\begin{thm}[Ito \cite{Ito:ss-reduction-K3-CM}*{Theorem 1.1}] \label{thm:reduction of CM K3}
Let $X$ be a K3 surface in characteristic $0$, 
having complex multiplication (CM) by a CM-field $E$,
and defined over a number field $K$ containing $E$.
Suppose $X$ has good reduction $X_v$ at a prime $v$ of $K$.
Let $\idealp$, $\idealq$, and $p$ be respectively the primes of $E$, $F$, and $\bQ$ below $v$,
where $F$ is the maximal totally-real subfield of $E$.
\begin{itemize}
\item
If $\idealq$ splits in $E$, then $X_v$ is of height $[E_{\idealp} : \bQ_p] < \infty$.
\item
If $\idealq$ does not split (in other words, if it ramifies or is inert) in $E$, then $X_v$ is supersingular (i.e.\ of height $\infty$).
\end{itemize}
\end{thm}
In this paper, we use this theorem only in the following situations.
\begin{itemize}
\item If $X$ has Picard number $20$,
then $X$ has CM by $E = \bQ(\sqrt{-\disc{T(X)}})$
(see \cite{Huybrechts:lecturesK3}*{Remark 3.3.10}),
where $T(X) = \Pic(X_{\bC})^{\perp} \subset H^2(X_{\bC}, \bZ)$ is the transcendental lattice.
In this case $F = \bQ$.

In this case Theorem \ref{thm:reduction of CM K3} is proved 
by Shimada \cite{Shimada:reduction of singular K3}*{Theorem 1} 
for all but finitely many $p$ not dividing $2 \disc \Pic(X)$ (for each $X$).
However we use Theorem \ref{thm:reduction of CM K3} for $p = 2$.
\item If $X$ admits an automorphism acting on $H^0(X, \Omega^2)$ 
by a primitive $m$-th root $\zeta_m$ of unity 
and if $\rank T(X) = \phi(m)$, 
then $X$ has CM by the cyclotomic field $E = \bQ(\zeta_m)$.
In this case $F = \bQ(\zeta_m + \zeta_m^{-1})$.

Jang \cite{Jang:representations-automorphism}*{Corollary 4.3} proved the following related result.
Suppose $Y$ is a K3 surface in characteristic $p > 2$
that admits an automorphism acting on $H^0(Y, \Omega^2)$ 
by a primitive $m$-th root of unity, 
and assume $22 - \phi(m) \leq \rho(Y)$. 
If $p^n \equiv -1 \pmod{m}$ for some $n$ then $Y$ is supersingular,
and otherwise $\height(Y)$ is equal to the order of $p$ in $(\bZ/m\bZ)^*$.
However we use Theorem \ref{thm:reduction of CM K3} for $p = 2$.
\end{itemize}

\section{Height of K3 surfaces and gap of morphisms between RDP K3 surfaces} \label{sec:height of morphisms}

In this section we prove the main results of the paper: Theorems \ref{thm:main:introduction} and \ref{thm:quotient:introduction} (Theorems \ref{thm:main}, \ref{thm:quotient}, and \ref{thm:quotient:etale}).

The height of a K3 surface $Y$ is characterized by the Frobenius action on $W_n$-valued cohomology $H^2(\Witt[n]{\cO_Y})$.
We regard this as a property of the Frobenius morphism
and generalize it to a property of morphisms between RDP K3 surfaces,
which we call the \emph{gap} of such morphisms.
Proposition \ref{prop:local to height:bis} shows how to compute the gap from the action on local cohomology groups. 
Using this, we derive main theorems from the local computations done in Section \ref{sec:RDP}.

In this section everything is in characteristic $p > 0$.

\subsection{Gap of morphisms between RDP K3 surfaces} \label{sec:height-RDP}
\begin{defn} \label{def:height of morphism}
Let $\map{\pi}{X}{Y}$ be a morphism between RDP K3 surfaces.
We define the \emph{gap} $\gap(\pi)$ of $\pi$
to be the minimum integer $n \geq 0$ such that 
the morphism $\map{\pi^*}{H^2(Y, \Witt[n+1]{\cO_Y})}{H^2(X, \Witt[n+1]{\cO_X})}$ is nonzero
if such $n$ exists,
and to be $\infty$ if there is no such $n$.
\end{defn}

\begin{rem} \label{rem;gap and height}
	By Theorem \ref{thm:vdGK},
	we have $\height(Y) = \gap(\Frob_Y) + 1$.
\end{rem}

We have the following.
\begin{lem} \label{lem:height}
	Let $\map{\pi}{X}{Y}$ be a morphism between RDP K3 surfaces.
\begin{enumerate}
	\item \label{lem:height:image of pi^*}
	The image of $\map{\pi^*}{H^2(Y, \Witt[n]{\cO_Y})}{H^2(X, \Witt[n]{\cO_X})}$ 
	is equal to $V^{\gap(\pi)}(H^2(X, \Witt[n-\gap(\pi)]{\cO_X}))$ if $n > \gap(\pi)$
	and to $0$ if $n \leq \gap(\pi)$.
	\item \label{lem:height:additive}
	If $\map{\pi'}{U}{X}$ is another morphism between RDP K3 surfaces,
	then $\gap(\pi \circ \pi') = \gap(\pi) + \gap(\pi')$
	(under the natural convention that $\infty + n = \infty + \infty = \infty$).
	\item \label{lem:height:birational}
	If $\pi$ is birational, then $\gap(\pi) = 0$.
	\item \label{lem:height:separable}
	More generally, if $\pi$ is dominant and separable, then $\gap(\pi) = 0$.
	\item \label{lem:height:linearly reductive}
	If $\pi$ is a $\mu_n$-quotient morphism, then $\gap(\pi) = 0$.
\end{enumerate}
\end{lem}	
\begin{proof}
	We will use the exact sequences in the proof of Lemma \ref{lem:Witt of RDP K3}.

	(\ref{lem:height:image of pi^*})
	For $n \leq \gap(\pi)$, this follows from the definition of $\gap(\pi)$.
	Suppose $n > \gap(\pi)$. 
	Write $\delta = \gap(\pi)$ and $\map{\pi^*_n = \pi^*}{H^2(\Witt[n]{\cO_Y})}{H^2(\Witt[n]{\cO_X})}$.
	Since $\pi^*_{\delta} = 0$, the morphism $\pi^*_n$ induces 
	$\map{\tilde{\pi}^*_n}{H^2(\Witt[n]{\cO_Y})}{H^2(\Witt[n - \delta]{\cO_X})}$ from the exact sequence.
	We will show by induction on $n \geq \delta$ that $\tilde{\pi}^*_n$ is surjective.
	This is clear if $n = \delta$.
	Suppose it is true for $n - 1$ ($\geq \delta$).
	Then since $\Image(\tilde{\pi}^*_n) \supset V(\Image(\tilde{\pi}^*_{n-1})) = V(H^2(\Witt[n - 1 - \delta]{\cO_X}))$
	and $H^2(\Witt[n - \delta]{\cO_X}) / V(H^2(\Witt[n - 1 - \delta]{\cO_X}))$ is of length $1$,
	either $\tilde{\pi}^*_n$ is surjective or $\Image(\tilde{\pi}^*_n) = V(H^2(\Witt[n - 1 - \delta]{\cO_X}))$.
	The latter would imply from the exact sequence that $\pi^*_{\delta + 1}$ is zero, contradicting the definition of $\delta$.

	(\ref{lem:height:additive})
	By (\ref{lem:height:image of pi^*}),
	the image of $\map{\pi'^* \circ \pi^*}
	{H^2(Y, \Witt[n]{\cO_Y})}{H^2(U, \Witt[n]{\cO_U})}$ 
	is 
	\[
	\begin{cases}
		0 & (n \leq \gap(\pi) + \gap(\pi')), \\
		V^{\gap(\pi) + \gap(\pi')}(H^2(U, \Witt[n - \gap(\pi) - \gap(\pi')]{\cO_U})) \neq 0 & (n > \gap(\pi) + \gap(\pi')).
	\end{cases}
	\]

	(\ref{lem:height:birational})
	By (\ref{lem:height:additive}),
	it suffices to consider the case where $\pi$ is the minimal resolution $\tilde{Y} \to Y$,
	and this follows from Lemma \ref{lem:Witt of RDP K3}.

	(\ref{lem:height:separable})
	For a general closed point $x \in X$, 
	the induced morphism $\cO_{Y, y} \to \cO_{X, x}$ is an isomorphism, where $y := \pi(x)$.
	Apply Proposition \ref{prop:local to height:bis}(\ref{prop:local to height:bis:upper}) to a generator of ${\localcohtorsion{y}{\cO_{Y,y}}{\idealm_y}}$ with $n = 1$.
	
	(\ref{lem:height:linearly reductive})
	Being the $\mu_n$-quotient, $\cO_Y \subset \pi_* \cO_X$ is a direct summand.
	It follows that $H^2(Y, \cO_Y) \to H^2(Y, \pi_* \cO_X) = H^2(X, \cO_X)$ is injective, hence nonzero.
\end{proof}

\begin{cor} \label{cor:height of separably isogenous K3}
	If a morphism $\map{\pi}{X}{Y}$ between RDP K3 surfaces has finite gap (e.g.\ if it is dominant and separable), then $\height(X) = \height(Y)$.
\end{cor}
\begin{proof}
Apply Lemma \ref{lem:height}(\ref{lem:height:additive}) to $\Frob_Y \circ \pi = \pi \circ \Frob_X$.
\end{proof}
In some cases we can bound $\gap(\pi)$
using behaviors of $\pi^*$ on local cohomology groups.

\begin{prop} \label{prop:local to height:bis}
	Let $\map{\pi}{X}{Y}$ be a morphism between RDP K3 surfaces.
	Let $x \in X$ be a point, $y = \pi(x)$, 
	$\idealm_y \subset \cO_{Y,y}$ the maximal ideal, and
	$I \subsetneq \cO_{Y,y}$ an $\idealm_y$-primary ideal.
	Let $n \geq 1$ be an integer. 
	Consider 
	an element $e \in \localcohtorsion{y}{\Witt[n]{\cO_{Y,y}}}{I}$
	and its image by the morphism	
	\[ \map{\pi^*}
	{\localcohtorsion{y}{\Witt[n]{\cO_{Y,y}}}{I}}
	{\localcohtorsion{x}{\Witt[n]{\cO_{X,x}}}{I \cO_{X,x}}} . \]

	\begin{enumerate}
		\item \label{prop:local to height:bis:upper}
		Suppose 
		$\pi^*(e) = V^{n-1}(e')$
		for some generator $e' \in \localcohtorsion{x}{\cO_{X,x}}{\idealm_x}$.
		Then $\gap(\pi) < n$.
		\item \label{prop:local to height:bis:lower}
		Suppose $I = \idealm_y$ and 
		$R^{n-1}(e) \in \localcohtorsion{y}{\cO_{Y,y}}{\idealm_y}$ is a generator.
		If $\pi^*(e) = 0$,
		then $\gap(\pi) \geq n$.
		\item \label{prop:local to height:bis:exact}
		Suppose $I = \idealm_y$, 
		$R^{n-1}(e) \in \localcohtorsion{y}{\cO_{Y,y}}{\idealm_y}$ is a generator, and
		$\pi^*(e) = V^{n-1}(e')$
		for some generator $e' \in \localcohtorsion{x}{\cO_{X,x}}{\idealm_x}$.
		Then $\gap(\pi) = n-1$.
	\end{enumerate}
\end{prop}

\begin{proof}
	Given $I \subsetneq \cO_{Y,y}$ and $n$ as in the statement,
	we consider the map $\map{\gamma = \gamma_{I,n}}{\localcohtorsion{y}{\Witt[n]{\cO_{Y,y}}}{I}}{H^2(Y, \Witt[n]{\cO_Y})}$ defined by
\[
	\localcohtorsion{y}{\Witt[n]{\cO_{Y,y}}}{I} 
	\isomfrom
	\Ext^2(\cO_{Y,y}/I, \Witt[n]{\cO_{Y,y}}) 
	\isomfrom \Ext^2(\cO_{Y}/\cI, \Witt[n]{\cO_{Y}}) \to
	 H^2(Y, \Witt[n]{\cO_Y}), 
\]
	where $\cI := \Ker(\cO_Y \to \cO_{Y,y} / I) \subset \cO_Y$
	(so $\Supp{\cO_Y / \cI} = \set{y}$ and $\cO_Y/\cI \cong \cO_{Y,y}/I$).
	This $\gamma = \gamma_{I,n}$ commutes 
	with inclusions of ideals $I' \subset I$,
	with $V$, 
	and with $\pi^*$ (with $\pi$ as in the statement). 

	If $n = 1$ and $I = \idealm_y$, then $\gamma_{\idealm_y,1}$ is an isomorphism,
	since the final map in the diagram is, by Serre duality,
	the dual of the isomorphism $H^0(Y, \cO_Y) \isomto H^0(Y, k(y))$.

	(\ref{prop:local to height:bis:upper})
	Applying $\gamma$ to $\pi^*(e) = V^{n-1}(e')$,
	we obtain $\pi^*(\gamma(e)) = V^{n-1}(\gamma(e'))$.
	As mentioned above, $\gamma_{\idealm_x,1}$ is an isomorphism, hence $V^{n-1}(\gamma(e'))$ is nonzero.
	Hence $\gap(\pi) \leq n-1$.
	
	(\ref{prop:local to height:bis:lower})
	By applying the induction hypothesis to $R(e)$ if $n > 1$,
	we obtain $\gap(\pi) \geq n-1$.
	Since $\gamma_{\idealm_y,1}$ is an isomorphism, $R^{n-1}(\gamma(e)) \in H^2(\cO_Y)$ is a generator.
	Hence $H^2(\Witt[n]{\cO_Y})$ is generated by $\gamma(e)$ and $V(H^2(\Witt[n-1]{\cO_Y}))$,
	and $\pi^*$ annihilates both since $\pi^*(e) = 0$ and $\gap(\pi) \geq n-1$.

	(\ref{prop:local to height:bis:exact})
	Applying (\ref{prop:local to height:bis:lower}) to $R(e)$ if $n > 1$,
	we obtain $\gap(\pi) \geq n-1$.
	Applying (\ref{prop:local to height:bis:upper}) to $e$, 
	we obtain $\gap(\pi) < n$.
\end{proof}

From this proposition we deduce bounds, or moreover exact values, of 
the height of RDP K3 surfaces with suitable singularities.
It turns out, surprisingly, that \emph{every} non-taut RDP is suitable.

\subsection{Gap of Frobenius maps and non-taut RDPs}

Combining Proposition \ref{prop:local to height:bis}
with the computation on Frobenius maps on local RDPs given in Section \ref{sec:RDP:Frobenius}, 
we prove the following relation between 
the isomorphism class of a non-taut RDP on an RDP K3 surface and the height of the surface.
This is trivially true when $r_{\max}(p, S) = 0$ (recall Convention \ref{conv:rmax=0}). 

\begin{thm}[Precise form of Theorem \ref{thm:main:introduction}] \label{thm:main}
	Let $S$ be a Dynkin diagram and $p \geq 0$ be a characteristic.
	Let $r_{\max} = r_{\max}(p,S)$ be the integer defined in Section \ref{sec:introduction}.
	Define a subsequence $(r_1, r_2, \dots, r_l)$ of $(r_{\max}(p,S), \dots, 2, 1)$ as follows.
	\begin{itemize}
		\item $(p, S) = (2, D_N)$, $N \geq 8$ ($r_{\max} = \floor{N/2} - 1$): 
		\begin{itemize}
			\item If $8 \leq N \leq 9$: $(r_1, r_2) = (\floor{N/2} - 1, \floor{N/2} - 2)$.
			\item If $10 \leq N $: $(r_1, r_2, r_3) = (\floor{N/2} - 1, \floor{N/2} - 2, \floor{N/2} - 4)$.
		\end{itemize}
		\item $(p, S) = (2, E_8)$ ($r_{\max} = 4$): $(r_1, r_2, r_3) = (4, 3, 2)$.
		\item all other cases: $(r_1, \dots, r_l)$ is the entire sequence $(r_{\max}(p,S), \dots, 2, 1)$.
	\end{itemize}
	Then we have the following.
	Suppose an RDP K3 surface $Y$ admits an RDP of type $S^r$.
	\begin{itemize}
		\item If $r > 0$, then $\height(Y) \leq l$ and $r = r_{\height(Y)}$.
		\item If $r = 0$, then $\height(Y) > l$.
	\end{itemize}
\end{thm}
\begin{cor} \label{cor:no E81}
	On RDP K3 surfaces in characteristic $2$, 
	RDPs of type 
	$D_N^r$ ($r > 0$ and $\floor{N/2} - r \not\in \set{1,2,4}$) and $E_8^1$ do not occur.
\end{cor}

\begin{proof}[Proof of Theorem \ref{thm:main}]
	Since $\height(Y) = \gap(\Frob_Y) + 1$ by Remark \ref{rem;gap and height}, 
	it suffices to compute $\delta := \gap(\Frob_Y)$.
	We apply Proposition \ref{prop:local to height:bis} to suitable elements of the local cohomology groups.

	Suppose $Y$ is an RDP K3 surface in characteristic $p$
	having a non-taut RDP of type $S^r$.
	If $(p, S^r) \neq (2, D_N^r), (2, E_8^1)$,
	then the assertion follows from Proposition \ref{prop:local to height:bis} 
	((\ref{prop:local to height:bis:exact}) if $r > 0$ and 
	 (\ref{prop:local to height:bis:lower}) if $r = 0$) applied 
	to the elements $e$ given in Proposition \ref{prop:local Frob:E:bis}.

	Suppose $(p, S^r) = (2, E_8^1)$.
	By Proposition \ref{prop:local to height:bis} 
	(\ref{prop:local to height:bis:upper}) and  
	(\ref{prop:local to height:bis:lower}) 
	applied to the elements given in 
	Propositions \ref{prop:local Frob:2:E8} and \ref{prop:local Frob:E:bis} respectively,
	we obtain $\delta \leq 0$ and $\delta \geq 3$.
	Contradiction. 

	Suppose $(p, S^r) = (2, D_N^r)$. 
	We have $r_{\max} + 1 = \floor{N/2}$.
	It suffices to show that 
	\begin{itemize}
	\item the inequality $\floor{N/2} - r \leq 2^{\delta}$ holds, 
	\item this inequality is equality if $r > 0$, and 
	\item $\delta \leq 2$ if $r > 0$.
	\end{itemize}
	Let $n'$ be the (unique) non-negative integer satisfying $2^{n'-1} < \floor{N/2} - r \leq 2^{n'}$.
	By applying Proposition \ref{prop:local to height:bis}(\ref{prop:local to height:bis:lower})
	to the element given in Proposition \ref{prop:local Frob:2:D} 
	for $(n, j) = (n', 1)$
	(in which case $a \geq 0$),
	we obtain $\delta \geq n'$.
	Hence $\floor{N/2} - r \leq 2^{n'} \leq 2^{\delta}$.

	Suppose moreover $r > 0$.
	Let $(n, j)$ be the (unique) pair of positive integers with $\floor{N/2} - r = 2^{n - 1} (2 j - 1)$.
	By applying Proposition \ref{prop:local to height:bis}(\ref{prop:local to height:bis:upper})
	to the element given in Proposition \ref{prop:local Frob:2:D} for this $(n,j)$
	(in which case $a = -1$),
	we obtain $\delta < n$. 
	Hence we have $2^{\delta} \leq 2^{n - 1} \leq 2^{n - 1} (2 j - 1) = \floor{N/2} - r$,
	therefore $\floor{N/2} - r = 2^{\delta}$.
	Since $N < 22 - 2 \height(Y) = 20 - 2 \delta$ (Corollary \ref{cor:bound of height})
	and $N \geq 2 (\floor{N/2} - r) = 2^{\delta + 1}$, 
	we have $\delta \leq 2$.
\end{proof}

\subsection{\texorpdfstring{Gap of $\mu_p$- and $\alpha_p$-quotient morphisms}{Gap of mu\_p- and alpha\_p-quotient morphisms}} \label{sec:height-mualpha}

Suppose 
$X$ and $Y$ are RDP K3 surfaces and 
$\map{\pi}{X}{Y}$ is a $G$-quotient morphism with $G \in \set{\mu_p, \alpha_p}$.
The author proved \cite{Matsumoto:k3alphap}*{Theorem 4.3}
that the ``dual'' map $\map{\pi'}{\thpower{Y}{1/p}}{X}$ is also a $G'$-quotient morphism with $G' \in \set{\mu_p, \alpha_p}$.
Here, both $G' = G$ and $G' \neq G$ are possible (see \cite{Matsumoto:k3alphap}*{Examples 10.2--10.4}).

\begin{defn}[\cite{Matsumoto:k3alphap}*{Definition 3.4}] \label{def:maximal}
	We say that a $G$-quotient morphism $\map{\pi}{X}{Y}$ between RDP K3 surfaces is \emph{maximal}
	if there is no point $x \in X$ such that $x$ and $\pi(x)$ are both RDPs.
\end{defn}
	The author proved \cite{Matsumoto:k3alphap}*{Corollary 3.5} that 
	for any $G$-quotient morphism $\map{\pi}{X}{Y}$ between RDP K3 surfaces 
	there is a maximal $G$-quotient morphism $\map{\pi_1}{X_1}{Y_1}$ between RDP K3 surfaces 
	with a birational and $G$-equivariant morphism $X_1 \to X$.
	Then $Y_1 \to Y$ is also birational, and hence $\gap(\pi) = \gap(\pi_1)$ by Lemma \ref{lem:height}(\ref{lem:height:additive},\ref{lem:height:birational}).

\begin{thm}[Precise form of Theorem \ref{thm:quotient:introduction}] \label{thm:quotient}
	Let $\map{\pi}{X}{Y}$ be as above.
\begin{enumerate}
\item \label{thm:quotient:each}
If $\pi$ is maximal (Definition \ref{def:maximal}),
then we have
\[
\gap(\pi) = \begin{cases}
0 & \text{if $G = \mu_p$ (in which case $p \leq 7$ and $\Sing(Y) = \frac{24}{p+1} A_{p-1}$),} \\
1 & \text{if $G = \alpha_p$ and $(p, \Sing(Y)) = (2, 2 D_4^0), (3, 2 E_6^0), (5, 2 E_8^0)$,} \\
2 & \text{if $G = \alpha_p$ and $(p, \Sing(Y)) = (2, 1 D_8^0)$,} \\
3 & \text{if $G = \alpha_p$ and $(p, \Sing(Y)) = (2, 1 E_8^0)$.}
\end{cases}
\]
This covers all possibilities for $G$, $p$, and $\Sing(Y)$ 
in the maximal case (\cite{Matsumoto:k3alphap}*{Theorem 4.6}).
\item \label{thm:quotient:composite}
We have $\height(X) = \height(Y) = \gap(\pi) + \gap(\pi') + 1$.
In particular, $X$ and $Y$ are of finite height.
\end{enumerate}
\end{thm}
\begin{proof}
(\ref{thm:quotient:each})
Let $y \in Y$ be a singular point.
Since $\pi$ is maximal, 
the inverse image $\pi^{-1}(y)$ of $y$ is smooth,
hence $\Spec \cO_{X,\pi^{-1}(y)} \to \Spec \cO_{Y,y}$ is as in Proposition \ref{prop:local mu_p alpha_p}.
Hence we obtain $\gap(\pi)$ from 
Proposition \ref{prop:local to height:bis}(\ref{prop:local to height:bis:exact}).
For the case $G = \mu_p$, we can also use Lemma \ref{lem:height}(\ref{lem:height:linearly reductive}).

(\ref{thm:quotient:composite})
	Since $\Frob_Y = \pi \circ \pi'$ and $\Frob_X = \pi' \circ \thpower{\pi}{1/p}$,
	the first assertion follows from Lemma \ref{lem:height}(\ref{lem:height:additive}).
	To show the finiteness of $\gap(\pi)$ and $\gap(\pi')$,
	we may assume $\pi$ is maximal, and then $\pi'$ is also maximal,
	and we can apply (\ref{thm:quotient:each}) to $\pi$ and $\pi'$.
\end{proof}

\begin{cor} \label{cor:no alpha alpha}
	Let $\map{\pi}{X}{Y}$ be as above.
	
If $p = 5$, then $(G, G') \neq (\alpha_5, \alpha_5)$.

If $p = 2$ and $\pi$ is maximal, then $(G, G', \Sing(X), \Sing(Y)) \neq (\alpha_2, \alpha_2, 1 E_8^0, 1 E_8^0)$.
\end{cor}

\begin{proof}
	We may suppose $\pi$ is maximal.
	(By above, this implies that if $p = 5$ and $G = \alpha_5$ then $\Sing(Y) = 2 E_8^0$.)
	Then the height of $Y$ asserted in Theorem \ref{thm:quotient}, which is $3$ or $7$ respectively, 
	contradicts Corollary \ref{cor:bound of height}.
\end{proof}

All other $(G, G', \Sing(X), \Sing(Y))$ is realizable
(see \cite{Matsumoto:k3alphap}*{Examples 10.2--10.5}).
Hence we have the following.

\begin{cor}
	Suppose an RDP K3 surface $X$ in characteristic $p$
	admits an action of $\mu_p$ or $\alpha_p$
	whose quotient is an RDP K3 surface.
	Then $\height(X) \leq 6,3,2,1$ for $p = 2,3,5,7$ respectively,
	and every such positive integer is realizable.
\end{cor}

Furthermore we have the following criterion.

\begin{cor} \label{cor:criterion}
	Suppose $X$ is an RDP K3 surface in characteristic $p$
	with a nontrivial $G$-action, $G \in \set{\mu_p, \alpha_p}$.
	Then,
\begin{itemize}
\item $\height(X) < \infty$ if and only if $X/G$ is an RDP K3 surface.
\item $\height(X) = \infty$ if and only if $X/G$ is either an RDP Enriques surface or a rational surface.
\end{itemize}	
\end{cor}
\begin{proof}
It is known (\cite{Matsumoto:k3alphap}*{Proposition 4.1})
that the quotient is either an RDP K3 surface, an RDP Enriques surface, or a rational surface.
	
We saw in Theorem \ref{thm:quotient} that if $X/G$ is an RDP K3 surface 
then $X$ is of finite height.

If $X/G$ is a rational surface or an RDP Enriques surface,
then $\Het^2(X/G, \bQ_l)$ is generated by algebraic cycles,
hence so is $\Het^2(X, \bQ_l)$,
hence $X$ is supersingular.
\end{proof}

\subsection{\texorpdfstring{The case of $\bZ/p\bZ$-quotients}{The case of Z/pZ-quotients}} \label{sec:height-etale}

Suppose $\map{\pi}{X}{Y}$ is a $\bZ/p\bZ$-quotient morphism between RDP K3 surfaces.
One can replace $X$ with its minimal resolution, to which the action extends, and this does not change the height of the surfaces.
Assuming $X$ is smooth, the author determined all possible configurations of singularities on $Y$
\cite{Matsumoto:k3alphap}*{Theorem 7.3(1)}.
In each case, the configuration contains a non-taut RDP with $r > 0$,
hence by Theorem \ref{thm:main} we can determine the height of $Y$.
By Corollary \ref{cor:height of separably isogenous K3}, we have $\height(X) = \height(Y)$.
Hence we obtain the following.
\begin{thm}[Precise form of Theorem \ref{thm:quotient:introduction}] \label{thm:quotient:etale}
	Let $\map{\pi}{X}{Y}$ be a $\bZ/p\bZ$-quotient morphism between RDP K3 surfaces.
	Then
	$\height(X)$ and $\height(Y)$ coincide, and are finite.
	Moreover, if $X$ is smooth then 
	\[
	\height(X) = \height(Y) = \begin{cases}
	1 & \text{if $(p, \Sing(Y)) = (2, 2 D_4^1), (3, 2 E_6^1), (5, 2 E_8^1)$}, \\
	2 & \text{if $(p, \Sing(Y)) = (2, 1 D_8^2)$}, \\
	3 & \text{if $(p, \Sing(Y)) = (2, 1 E_8^2)$}.
	\end{cases}
	\]
	This covers all possibilities for $p$ and $\Sing(Y)$ in the smooth case 
	(\cite{Matsumoto:k3alphap}*{Theorem 7.3(1)}).
\end{thm}

\begin{rem}
	In the case of $\bZ/p\bZ$-quotients, we do not have an equivalence as in Corollary \ref{cor:criterion}.
	There is an example of an $\bZ/p\bZ$-action on an ordinary K3 surface $X$ 
	with rational or Enriques quotient $Y$, at least in characteristic $2$.
\end{rem}

\section{RDPs realizable on K3 surfaces} \label{sec:realizable}

We determine which RDPs can occur on K3 surfaces.

\subsection{Non-taut RDPs}

In the non-taut case,
Theorem \ref{thm:main} (and Corollary \ref{cor:no E81}) and Corollary \ref{cor:bound of height}
give necessary conditions.
We will show in Proposition \ref{prop:no D19^8} that $D_{19}^8$ in characteristic $2$ is impossible.
It turns out that all remaining RDPs are realizable on RDP K3 surfaces,
as we will see in Section \ref{sec:examples}.
Summarizing:
\begin{thm} \label{thm:existence:non-taut}
	Consider a non-taut RDP $D_N^r$ or $E_N^r$ in characteristic $p > 0$.
	Then it occurs on some RDP K3 surface $Y$ in characteristic $p$ 
	if and only if it satisfies the following conditions.
	\begin{itemize}
		\item It does not contradict Corollary \ref{cor:no E81} and Proposition \ref{prop:no D19^8}
		(i.e.\ if $p = 2$, then it is not $D_N^r$ with $r > 0$ and $\floor{N/2} - r \notin \set{1,2,4}$, 
		nor $D_{19}^8$, nor $E_8^1$).
		\item $N < 22 - 2 h$ if $r > 0$, where $h$ is the height predicted in Theorem \ref{thm:main}.
		$N < 22$ if $r = 0$.
	\end{itemize}
\end{thm}
\begin{prop} \label{prop:no D19^8}
	An RDP K3 surface in characteristic $2$
	cannot have an RDP of type $D_{19}^8$.
\end{prop}

\begin{proof}[Proof of Proposition \ref{prop:no D19^8}]
	Suppose $z \in Y$ is an RDP of type $D_{19}^8$ in characteristic $2$ on an RDP K3 surface $Y$.
	By Theorem \ref{thm:main}, $\height(Y) = 1$.
	Let $\tilde{Y} \to Y$ be the minimal resolution.
	Since $\height(\tilde{Y}) < \infty$ there exists, 
	by \cite{Lieblich--Maulik:cone}*{Corollary 4.2},
	a K3 surface $\cX$ over $\Spec W(k)$ with $\cX \otimes_{W(k)} k \cong \tilde{Y}$
	and $\Pic(\cX) \isomto \Pic(\tilde{Y})$.
	Let $X_{K} := \cX \otimes_{W(k)} {K}$ be the generic fiber of $\cX$ 
	over $K := \Frac W(k)$
	and let $X_{\bC} := X_K \otimes_{K} \bC$ for any embedding $K \to \bC$ 
	(which we may assume to exist by replacing $k$).
	Then we have $\Pic(X_{\bC}) \cong \Pic(\cX) \cong \Pic(\tilde{Y})$.
	
	Let $L_{1}$ be the sublattice of $\Pic(\tilde{Y}) \cong \Pic(X_{\bC})$ generated by the exceptional curves above $z$,
	and $L_2 := L_{1}^{\perp}$ be its orthogonal complement.
	Since $L_{1}$ is negative definite, $L_2$ is nonzero, 
	and since $\rho(\tilde{Y}) \leq 22 - 2 \height(\tilde{Y}) = 20$ (Proposition \ref{prop:bound of height}), 
	we have $\rank L_2 = 1$.
	The transcendental lattice $T = T(X_{\bC}) = (\Pic(X_{\bC}))^{\perp}$ in $H^2(X_{\bC}, \bZ)$ of $X_{\bC}$ 
	is a rank $2$ positive definite lattice,
	and then $X_{\bC}$ has complex multiplication by the imaginary quadratic field 
	$E := \bQ(\sqrt{-d})$, $d := \disc T(X_{\bC})$. 
	By Theorem \ref{thm:reduction of CM K3}, 
	the reduction $\tilde{Y}$ of $X_{\bC}$ at a prime above $2$ being ordinary 
	implies that $2$ is split in $E/\bQ$.
	Writing $d = k^2 d_0$ with $d_0$ square-free, 
	this means $d_0 \equiv -1 \pmod{8}$.
	By Lemma \ref{lem:lattice}, this is impossible.
\end{proof}
\begin{lem} \label{lem:lattice}
	Suppose $L_1$, $L_2$, and $L_3$ are lattices with
	\begin{itemize}
		\item $\disc(L_1) = - 4^x$ for some non-negative integer $x$,
		\item $L_2$ is positive definite, $\rank(L_2) = 1$,
		\item $L_3$ is positive definite, $\rank(L_3) = 2$, 
		$\disc(L_3) = k^2 d_0$ with $d_0$ square-free and $d_0 \equiv -1 \pmod{8}$.
	\end{itemize}
	Then $L_1 \oplus L_2 \oplus L_3$ does not admit a unimodular overlattice of finite index. 
\end{lem}
A non-degenerate lattice $L$ is called unimodular if 
the natural injection $L \injto L^* := \Hom(L, \bZ)$ is an isomorphism, 
equivalently if $\disc(L) = \pm 1$.

\begin{proof}
	Suppose $L_1 \oplus L_2 \oplus L_3$ admits a finite index overlattice $\Lambda$ 
	with $\disc(\Lambda) = \pm 1$.
	Take bases $e_2$ of $L_2$ and $t_1, t_2$ of $L_3$,
	and let $(m)$ and $\begin{pmatrix} a & b \\ b & c \end{pmatrix}$ be the Gram matrices 
	(so $m > 0$, $a > 0$, and $\disc(L_3) = k^2 d_0 = a c - b^2 > 0$).
	Since $L_1 \oplus L_2 \oplus L_3 \subset \Lambda$ is finite index, 
	its discriminant $\disc(L_1 \oplus L_2 \oplus L_3) = \disc(L_1) \cdot \disc(L_2) \cdot \disc(L_3)
	= -4^x m k^2 d_0$ coincides with $\disc(\Lambda) = \pm 1$ up to a square.
	Hence $m = n^2 d_0$.

	Let $g = \gcd\set{a,b,c}$. 
	Then the discriminant group of $L_3$ is isomorphic to $\bZ/g\bZ \times \bZ/gh \bZ$,
	where $h = g^{-2}\disc(L_3) \in \bZ$.
	By lattice theory this is isomorphic to the discriminant group 
	of the primitive closure 
	of $L_1 \oplus L_2$ in $\Lambda$,
	which is a subquotient of the discriminant group of $L_1 \oplus L_2$.
	Hence $g$ is a power of $2$.
	
	We claim that $a$ is the norm of some ideal of $\cO_E$,
	where $E = \bQ(\sqrt{-d_0})$.
	It suffices to show that $\ord_l(a)$ is even for any prime $l$ that is inert in $E/\bQ$.
	Suppose $\ord_l(a) = 2j - 1$ and $l$ is inert (then $l \neq 2$ since $-d_0 \equiv 1 \pmod{8}$). 
	Then, since $ac = b^2 + k^2 d_0$ and since $l$ is inert, 
	we have $\ord_l(k^2 d_0) \geq 2j$ and $\ord_l(b^2) \geq 2j$, 
	hence $l \divides c$, hence $l \divides g$, hence $l = 2$. Contradiction.
	
	Since $\Lambda$ is unimodular, there is an element $v \in \Lambda$ with $e_2 \cdot v = 1$.
	Write $2^x n v = v_1 + v_2 + v_3$ with 
	$v_i \in L_i \otimes \bQ$,
	then $v_i \in 2^x n L_i^*$.
	We have $v_2 = (2^x n / m) e_2$ and hence $v_2^2 = (2^x n / m)^2 m = 4^x / d_0$.
	We have $v_1^2 \in \bZ$
	(since $v_1 \in 2^x n L_1^*$ and $L_1^* \subset \abs{\disc(L_1)}^{-1} L_1 = 4^{-x} L_1$).
	We have $\sum v_i^2 = (2^x n)^2 v^2 \in \bZ$.
	Hence we obtain $v_3^2 \equiv - 4^x/d_0 \pmod{\bZ}$, 
	hence $d_0 v_3^2 \equiv - 4^x \pmod{d_0 \bZ}$.

	Write $v_3 = x_1 t_1 + x_2 t_2$ ($x_i \in \bQ$).
	Then $d_0 = N_{E/\bQ}(\sqrt{-d_0})$ and
	$a v_3^2 = a (a x_1^2 + 2 b x_1 x_2 + c x_2^2) = N_{E/\bQ}(a x_1 + (b + \sqrt{-k^2 d_0}) x_2)$
	are the norms of elements of $E$.
	Hence $d_0 v_3^2 = d_0 \cdot a^{-1} \cdot a v_3^2$ 
	is the norm of a fractional ideal of $E$.
	Therefore $-4^x$ and hence $-1$ are norm residues modulo $d_0$.
	But $-1$ cannot be a norm residue of an imaginary quadratic field. 
	Contradiction.
\end{proof}

\subsection{Taut RDPs}

For the taut case we have the following,
which is almost done by Shimada and Shimada--Zhang.

\begin{thm} \label{thm:existence:taut}
Suppose $p \geq 0$.
Suppose $S$ is a Dynkin diagram ($A_N$, $D_N$, $E_N$)
for which RDPs of type $S$ in characteristic $p$ are taut.
Then such an RDP occurs on some RDP K3 surface $Y$ in characteristic $p$
if and only if $p$ satisfies the following respective conditions.
\begin{itemize}
	\item If $N \leq 19$: any $p \geq 0$. 
	\item If $S$ is $A_{20}$: 
	$p > 0$ and $p$ is non-split in $\bQ(\sqrt{21})$.
	Equivalently, either $p \divides 21$ or $p \equiv \pm 2, \pm 8, \pm 10 \pmod{21}$.
	\item If $S$ is $A_{21}$: $p = 11$.
	\item If $S$ is $D_{20}$ or $D_{21}$, or $N \geq 22$: no $p$.
\end{itemize}
\end{thm}

\begin{proof}
Suppose $N \leq 19$ and $p \neq 2$.
It is known 
that there exists an elliptic K3 surface with a section and a singular fiber of type $\rI_{19}$.
For $p = 0$ this is due to Shioda \cite{Shioda:maximal singular fiber}*{Theorem 1.1} 
(who used the equation given by Hall \cite{Hall:Diophantine}*{equation 4.29 in page 185}).
It is clear from the Shioda's equation that the same equation in characteristic $p > 3$ also gives 
an elliptic K3 surface with the same property.
Moreover this holds for $p = 3$ using the coordinate change given by Sch\"utt--Top \cite{Schutt--Top}*{Section 2}.
Then the union of a section and this singular fiber contains a configuration of type $S$.

Suppose $N \leq 20$ and $p = 2$.
Then $S$ is a subset of $D_{21}$,
which is realized by Theorem \ref{thm:existence:non-taut} (Example \ref{ex:p=2}).

Suppose $S$ is $A_{20}$ and $p \neq 2$.
If $p \notdivides 2 \disc(A_{20}) = 2 \cdot 3 \cdot 7$,
then by \cite{Shimada--Zhang:rank 20}*{Table}\footnote{The table is contained only in the preprint version available at Shimada's website. },
this is possible if and only if $\Legendre{21}{p} = -1$.
If $p = 7$, then this is possible by \cite{Shimada:RDP on ssK3}*{Table RDP}.
It remains to show that it is possible if $p = 3$.
Let $L$ be the Dynkin lattice of type $A_{20}$ 
and $T$ be the lattice of rank $2$ with basis $t_1,t_2$ and Gram matrix 
$\begin{pmatrix} 2 & 5 \\ 5 & 2 \end{pmatrix}$.
Let $e_1, e_2, \dots, e_{20}$ be a basis of $A_{20}$ with 
$e_i \cdot e_j = -2,1,0$ if $\abs{i-j} = 0$, $\abs{i-j} = 1$, $\abs{i-j} \geq 2$ respectively.
We have $L^*/L \cong \bZ/21\bZ$ and $T^*/T \cong \bZ/21\bZ$.
Let $l = \frac{1}{7}(\sum_{i = 1}^{20} i e_i) \in L^*$
and $t = \frac{4}{7}(t_1 + t_2) \in T^*$.
They generate the prime-to-$3$ parts of $L^*/L$ and $T^*/T$ respectively.
We have $l^2 \equiv - t^2 \pmod{2\bZ}$
since $l^2 + t^2 = (\frac{1}{7})^2 \cdot (-20 \cdot (20+1)) + (\frac{4}{7})^2 \cdot 14 
= -4 \in 2 \bZ$.
We can apply Lemma \ref{lem:reduction to lattice} below.

Suppose $S$ is $A_{21}$.
Then $Y$ is supersingular and, considering the Picard lattice,
we must have $p \divides 22$.
By \cite{Shimada:RDP on ssK3}*{Table RDP},
this is possible for $p = 11$ and impossible for $p = 2$.

Suppose $S$ is $D_{20}$.
Since $\disc(S) = 4$ is a square,
an RDP of type $S$ can be realized only in characteristic $p$ dividing $\disc(S)$ by \cite{Dolgachev--Keum:auto}*{Lemma 3.2},
that is, $p = 2$.
In this case $S$ is non-taut and is out of the scope of this theorem.
\end{proof}

For a finite abelian group $A$,
we write its $p$-primary part (resp.\ prime-to-$p$ part) by $A_{p}$ (resp.\ $A_{p'}$).
For a non-degenerate even lattice $L$,
we define a quadratic map $\map{q_L}{L^*/L}{\bQ/2\bZ}$
by $q_L(\bar{v}) = v^2 \bmod {2\bZ}$,
where the bar denotes the projection $L^* \to L^*/L$. 
The next lemma is a variant of \cite{Shimada--Zhang:rank 20}*{Proposition 2.6}.
\begin{lem} \label{lem:reduction to lattice}
Let $p$ be an odd prime.
Let $R$ be a formal finite sum of $A_N$, $D_N$, $E_N$, with $\sum N = 20$,
and let $L = L(R)$ be the corresponding lattice (of rank $20$).
Suppose there are an even lattice $T$ of sign $(+1,-1)$ and
a group isomorphism $\map[\isomto]{\phi}{(L^*/L)_{p'}}{(T^*/T)_{p'}}$ 
satisfying
$\phi^*(q_T \restrictedto{(T^*/T)_{p'}}) = - q_L \restrictedto{(L^*/L)_{p'}}$
and $(L^*/L)_{p} \oplus (T^*/T)_{p} \cong (\bZ/p\bZ)^2$.
Then there exists an RDP K3 surface $Y$ (supersingular of Artin invariant $1$) 
with $\Sing(Y) = R$.
\end{lem}
\begin{proof}
Let $\Lambda$ be the submodule of $L^* \oplus T^*$ 
consisting of the elements $(l,t)$ with 
$\bar{l} \in (L^*/L)_{p'}$, 
$\bar{t} \in (T^*/T)_{p'}$, 
and $\phi(\bar{l}) = \bar{t}$.
Then $\Lambda$ is an even overlattice of $L \oplus T$ of sign $(+1,-21)$
with $\Lambda^*/\Lambda \cong (\bZ/p\bZ)^2$.
This means that $\Lambda$ is isomorphic to the Picard lattice of a supersingular K3 surface of Artin invariant $1$.
By the argument of \cite{Shimada--Zhang:rank 20}*{Theorem 2.1, (3) $\implies$ (1)},
we obtain a supersingular RDP K3 surface $Y$ of Artin invariant $1$
with $\Sing(Y) = R$.
\end{proof}

\section{Examples} \label{sec:examples}

Examples of maximal $G$-quotient morphisms $X \to Y$ between RDP K3 surfaces
with all possible $(G, G', \Sing(X), \Sing(Y))$
are already given in \cite{Matsumoto:k3alphap}*{Examples 10.2--10.5}.

The non-taut RDPs in Examples \ref{ex:p=2}--\ref{ex:p=5},
together with their partial resolutions,
prove the existence part of Theorem \ref{thm:existence:non-taut}.

\begin{example}[$p = 2$] \label{ex:p=2}
	\ 
	\begin{itemize}
	\item \label{item:h=1}
	Sch\"utt \cite{Schutt:maximal singular fiber}*{Section 6.2} gave an example of an elliptic K3 surface
	\[
	y^2 + t x y + t^6 y = x^3 + (c^2 t^4 + c t^3 + \tilde{a}_6) x^2 + c t^8 x + t^{10} \tilde{a}_6,
	\]
	where $\tilde{a}_6 \in k[t]$ is of degree $\leq 2$,
	with a section and a singular fiber of type $\rI^*_{13}$.
	It is of height $1$ since the coefficient of $t x y$ is nonzero 
	(Proposition \ref{prop:height:wxyz} applied to $\bP(6,4,1,1)$).
	The union of a section and this singular fiber contains a configuration of type $D_{18}$ and one of type $E_8$.
	The respective contractions give RDP K3 surfaces with $D_{18}^r$ and $E_8^{r'}$.
	By Theorem \ref{thm:main}, we have $r = r_{\max}(2, D_{18}) = 8$ and $r' = r_{\max}(2, E_8) = 4$. 
	\item \label{item:h=2}
	Let $E_1$ and $E_2$ be elliptic curves, ordinary and supersingular respectively.
	Let $X = \Km (E_1 \times E_2)$, i.e.\ $X$ is the minimal resolution of $(E_1 \times E_2) / \set{\pm 1}$.
	By \cite{Shioda:Kummer2}*{Section 6(b)} and \cite{Artin:wild2}*{Examples},
	$\Sing(X)$ is $2 D_8^2$. 
	Hence by Theorem \ref{thm:main}, $\height(X) = 2$.
	Consider the elliptic fibrations 
	$f_j \colon X = \Km (E_1 \times E_2) \to (E_1 \times E_2) / \set{\pm 1} \to E_j / \set{\pm 1} \cong \bP^1$.
	They admit sections and, by \cite{Shioda:Kummer2}*{Section 4},
	the singular fibers of $f_1$ and $f_2$ are $2 \rII^*$ and $1 \rI_{12}^*$ respectively.
	In either case, the union of a section and the singular fiber(s) contains a configuration of type $D_{17}$ and one of type $E_8$.
	\item \label{item:h=3} 
	Let $X$ be the elliptic RDP K3 surface 
	$y^2 + y x t^2 + x^3 + t^5 = 0$, 
	$y'^2 + y' x' + x'^3 + u^7 = 0$
	($u = t^{-1}$, $x' = t^{-4} x$, $y' = t^{-6} y$).
	The singular fibers of its minimal resolution are $\rII^*$ and $\rI_7$, hence
	the union of a section and the singular fibers contains a configuration of type $D_{15}$ and one of type $E_8$.
	We have two proofs for $\height(X) = 3$.
	(1) In this case, it is clear that the RDP at $(t = x = y = 0)$ is of type $E_8^2$. 
	Apply Theorem \ref{thm:main}.
	(2) Counting $\#(X(\bF_{2^n}))$ (before taking the resolution),
	we obtain 
	$\# X(\bF_2) = 1 + 2^2 + 2 \cdot 2$,
	$\# X(\bF_4) = 1 + 4^2 + 4 \cdot 2$, and
	$\# X(\bF_8) = 45 = 1 + 8^2 + 8 \cdot (-5/2)$,
	hence $\height(X) = 3$ by Proposition \ref{prop:height:point counting}
	($s(1) = 2, s(2) = 1, s(3) = -3/2$).
	(3) 
	Let $\tilde{X}$ be the (smooth) elliptic K3 surface in characteristic $0$ defined by the same equation.
	Since $\tilde{X}$ admits an automorphism ($(x',y',u) \mapsto (x',y',\zeta_7 u)$)
	acting on $H^0(\tilde{X}, \Omega^2)$ by a primitive $7$-th root of unity,
	$\tilde{X}$ has complex multiplication by $\bQ(\zeta_7)$.
	Hence the mod $2$ reduction $X$ of $\tilde{X}$ has $\height(X) = 3$ 
	by Theorem \ref{thm:reduction of CM K3}. 
	\item \label{item:h=infty} 
	The quasi-elliptic K3 surface $y^2 = x^3 + t^2 x + t^{11}$
	(given by Dolgachev--Kondo \cite{Dolgachev--Kondo:supersingular}*{Theorem 1.1})
	admits a fiber of type $\rI_{16}^*$ at $t = 0$.  
	The union of a section and the singular fibers contains a configuration of type $D_{21}$
	and one of type $E_8$.
	Since $21 \not< 22 - 2h$ for any $1 \leq h < \infty$, this K3 surface is supersingular.
	\end{itemize}
\end{example}

\begin{example}[$p = 3$] \label{ex:p=3}
	\ 
	\begin{enumerate}
	\item 
	$X \colon y^2 + x^3 + t^2 x^2 + t^5 + t^6 + t^7 = 0$.
	$\Sing(X) =  2 E_8^2 + A_2$.
	\item 
	$X \colon y^2 + x^3 + t^3 x^2 + t^5 = 0$.
	$\Sing(X) =  E_8^1 + D_8$.
	\item 
	$X \colon y^2 + x^3 + t^5 + t^7 = 0$.
	$\Sing(X) = 2 E_8^0 + 2 A_2$.
	\end{enumerate}
\end{example}

\begin{example}[$p = 5$] \label{ex:p=5}
	As in \cite{Matsumoto:k3alphap}*{Example 10.11},
	$y^2 = x^3 + a t^4 x + t + t^{11}$
	is an RDP K3 surface with $2 E_8^1$ (resp.\ $2 E_8^0$) if $a \neq 0$ (resp.\ $a = 0$).
\end{example}

\subsection*{Acknowledgments}
I thank 
Hiroyuki Ito, Kazuhiro Ito, Teruhisa Koshikawa, Hisanori Ohashi, and Fuetaro Yobuko 
for helpful comments and discussions.
I thank the anonymous referee for suggestions and corrections.


\begin{bibdiv}
	\begin{biblist}

\bib{Artin:wild2}{article}{
  author={Artin, M.},
  title={Wildly ramified $Z/2$ actions in dimension two},
  journal={Proc. Amer. Math. Soc.},
  volume={52},
  date={1975},
  pages={60--64},
  issn={0002-9939},
}

\bib{Artin:RDP}{article}{
  author={Artin, M.},
  title={Coverings of the rational double points in characteristic $p$},
  conference={ title={Complex analysis and algebraic geometry}, },
  book={ publisher={Iwanami Shoten, Tokyo}, },
  date={1977},
  pages={11--22},
}

\bib{Artin--Mazur:formalgroups}{article}{
  author={Artin, M.},
  author={Mazur, B.},
  title={Formal groups arising from algebraic varieties},
  journal={Ann. Sci. \'{E}cole Norm. Sup. (4)},
  volume={10},
  date={1977},
  number={1},
  pages={87--131},
  issn={0012-9593},
}

\bib{Borger:WittII}{article}{
  author={Borger, James},
  title={The basic geometry of Witt vectors. II: Spaces},
  journal={Math. Ann.},
  volume={351},
  date={2011},
  number={4},
  pages={877--933},
  issn={0025-5831},
}

\bib{Dolgachev--Keum:auto}{article}{
  author={Dolgachev, Igor V.},
  author={Keum, JongHae},
  title={Finite groups of symplectic automorphisms of $K3$ surfaces in positive characteristic},
  journal={Ann. of Math. (2)},
  volume={169},
  date={2009},
  number={1},
  pages={269--313},
  issn={0003-486X},
}

\bib{Dolgachev--Kondo:supersingular}{article}{
  author={Dolgachev, I.},
  author={Kond\={o}, S.},
  title={A supersingular $K3$ surface in characteristic 2 and the Leech lattice},
  journal={Int. Math. Res. Not.},
  date={2003},
  number={1},
  pages={1--23},
  issn={1073-7928},
}

\bib{vanderGeer--Katsura:stratification}{article}{
  author={van der Geer, G.},
  author={Katsura, T.},
  title={On a stratification of the moduli of $K3$ surfaces},
  journal={J. Eur. Math. Soc. (JEMS)},
  volume={2},
  date={2000},
  number={3},
  pages={259--290},
  issn={1435-9855},
}

\bib{Hall:Diophantine}{article}{
  author={Hall, Marshall, Jr.},
  title={The Diophantine equation $x^{3}-y^{2}=k$},
  conference={ title={Computers in number theory}, address={Proc. Sci. Res. Council Atlas Sympos. No. 2, Oxford}, date={1969}, },
  book={ publisher={Academic Press, London}, },
  date={1971},
  pages={173--198},
}

\bib{Hartshorne:localcohomology}{book}{
  author={Hartshorne, Robin},
  title={Local cohomology},
  series={A seminar given by A. Grothendieck, Harvard University, Fall},
  volume={1961},
  publisher={Springer-Verlag, Berlin-New York},
  date={1967},
  pages={vi+106},
}

\bib{Hartshorne:AG}{book}{
  author={Hartshorne, Robin},
  title={Algebraic geometry},
  note={Graduate Texts in Mathematics, No. 52},
  publisher={Springer-Verlag},
  place={New York},
  date={1977},
  pages={xvi+496},
  isbn={0-387-90244-9},
}

\bib{Hazewinkel:formalgroups}{book}{
  author={Hazewinkel, Michiel},
  title={Formal groups and applications},
  note={Corrected reprint of the 1978 original},
  publisher={AMS Chelsea Publishing, Providence, RI},
  date={2012},
  pages={xxvi+573},
  isbn={978-0-8218-5349-8},
}

\bib{Huybrechts:lecturesK3}{book}{
  author={Huybrechts, Daniel},
  title={Lectures on K3 Surfaces},
  date={2016},
  eprint={http://www.math.uni-bonn.de/people/huybrech/K3.html},
  publisher={Cambridge University Press, Cambridge},
  series={Cambridge Studies in Advanced Mathematics},
  volume={158},
}

\bib{Illusie:report}{article}{
  author={Illusie, Luc},
  title={Report on crystalline cohomology},
  conference={ title={Algebraic geometry}, address={Proc. Sympos. Pure Math., Vol. 29, Humboldt State Univ., Arcata, Calif.}, date={1974}, },
  book={ publisher={Amer. Math. Soc., Providence, R.I.}, },
  date={1975},
  pages={459--478},
}

\bib{Illusie:deRham--Witt}{article}{
  author={Illusie, Luc},
  title={Complexe de de\thinspace Rham-Witt et cohomologie cristalline},
  language={French},
  journal={Ann. Sci. \'{E}cole Norm. Sup. (4)},
  volume={12},
  date={1979},
  number={4},
  pages={501--661},
  issn={0012-9593},
}

\bib{Ito:ss-reduction-K3-CM}{article}{
  author={Ito, Kazuhiro},
  title={On the Supersingular Reduction of $K3$ Surfaces with Complex Multiplication},
  journal={International Mathematics Research Notices},
  date={2018},
  month={09},
}

\bib{Jang:representations-automorphism}{article}{
  author={Jang, Junmyeong},
  title={The representations of the automorphism groups and the Frobenius invariants of K3 surfaces},
  journal={Michigan Math. J.},
  volume={65},
  date={2016},
  number={1},
  pages={147--163},
}

\bib{Lieblich--Maulik:cone}{article}{
  author={Lieblich, Max},
  author={Maulik, Davesh},
  title={A note on the cone conjecture for K3 surfaces in positive characteristic},
  journal={Math. Res. Lett.},
  volume={25},
  date={2018},
  number={6},
  pages={1879--1891},
  issn={1073-2780},
}

\bib{Liedtke--Martin--Matsumoto:RDPtors}{article}{
  author={Liedtke, Christian},
  author={Martin, Gebhard},
  author={Matsumoto, Yuya},
  title={Torsors over the rational double points in characteristic $p$},
  year={2021},
  eprint={https://arxiv.org/abs/2110.03650v1},
}

\bib{Matsumoto:k3alphap}{article}{
  author={Matsumoto, Yuya},
  title={$\mu _p$- and $\alpha _p$-actions on K3 surfaces in characteristic $p$},
  year={2022},
  eprint={https://arxiv.org/abs/1812.03466v5},
  journal={J. Algebraic Geom.},
  status={to appear},
}

\bib{Schutt:maximalsingularfiber}{article}{
  author={Sch\"{u}tt, Matthias},
  title={The maximal singular fibres of elliptic $K3$ surfaces},
  journal={Arch. Math. (Basel)},
  volume={87},
  date={2006},
  number={4},
  pages={309--319},
  issn={0003-889X},
}

\bib{Schutt--Top}{article}{
  author={Sch\"{u}tt, Matthias},
  author={Top, Jaap},
  title={Arithmetic of the $[19,1,1,1,1,1]$ fibration},
  journal={Comment. Math. Univ. St. Pauli},
  volume={55},
  date={2006},
  number={1},
  pages={9--16},
  issn={0010-258X},
}

\bib{Shimada:RDPonssK3}{article}{
  author={Shimada, Ichiro},
  title={Rational double points on supersingular $K3$ surfaces},
  journal={Math. Comp.},
  volume={73},
  date={2004},
  number={248},
  pages={1989--2017},
  issn={0025-5718},
}

\bib{Shimada:reductionofsingularK3}{article}{
  author={Shimada, Ichiro},
  title={Transcendental lattices and supersingular reduction lattices of a singular $K3$ surface},
  journal={Trans. Amer. Math. Soc.},
  volume={361},
  date={2009},
  number={2},
  pages={909--949},
  issn={0002-9947},
}

\bib{Shimada--Zhang:rank20}{article}{
  author={Shimada, Ichiro},
  author={Zhang, De-Qi},
  title={Dynkin diagrams of rank 20 on supersingular $K3$ surfaces},
  journal={Sci. China Math.},
  volume={58},
  date={2015},
  number={3},
  pages={543--552},
  issn={1674-7283},
}

\bib{Shioda:Kummer2}{article}{
  author={Shioda, Tetsuji},
  title={Kummer surfaces in characteristic $2$},
  journal={Proc. Japan Acad.},
  volume={50},
  date={1974},
  pages={718--722},
  issn={0021-4280},
}

\bib{Shioda:maximalsingularfiber}{article}{
  author={Shioda, Tetsuji},
  title={The elliptic $K3$ surfaces with a maximal singular fibre},
  language={English, with English and French summaries},
  journal={C. R. Math. Acad. Sci. Paris},
  volume={337},
  date={2003},
  number={7},
  pages={461--466},
  issn={1631-073X},
}

\bib{Takagi--Watanabe:F-singularities}{article}{
  author={Takagi, Shunsuke},
  author={Watanabe, Kei-Ichi},
  title={$F$-singularities: applications of characteristic $p$ methods to singularity theory},
  journal={Sugaku Expositions},
  volume={31},
  date={2018},
  number={1},
  pages={1--42},
  issn={0898-9583},
}

\bib{Tanaka:taut}{article}{
  author={Tanaka, Yuki},
  title={On tautness of two-dimensional $F$-regular and $F$-pure rational singularities},
  year={2015},
  eprint={https://arxiv.org/abs/1502.07236v1},
}




	\end{biblist}
\end{bibdiv}
\end{document}